\documentclass[12pt]{amsart}
\usepackage{amssymb}
\usepackage{graphicx}

\newcommand{\re}{{\rm Re\,}}
\newcommand{\im}{{\rm Im\,}}
\newcommand{\eps}{\varepsilon}

\newcommand{\N}{{\mathbb N}}
\newcommand{\C}{{\mathbb C}}
\newcommand{\Z}{{\mathbb Z}}
\newcommand{\R}{{\mathbb R}}

\newcommand{\tef}{transcendental entire function}

\newcommand{\B}{\mathcal{B}}
\newcommand{\sing}{\operatorname{sing}}
\newcommand{\sw}{spider's web}

\theoremstyle{plain}
\newtheorem{theorem}{Theorem}[section]

\newtheorem{lemma}[theorem]{Lemma}     
\newtheorem{corollary}[theorem]{Corollary}
\theoremstyle{definition}
\newtheorem{definition}{Definition}[section]
\theoremstyle{remark}

\theoremstyle{problem}

\theoremstyle{example}
\newtheorem{example}{Example}

\begin{document}


\title[Fast escaping points]{Fast escaping points of entire functions}

\author{P.J. Rippon}
\address{Department of Mathematics and Statistics \\
The Open University \\
   Walton Hall\\
   Milton Keynes MK7 6AA\\
   UK}
\email{p.j.rippon@open.ac.uk}

\author{G.M. Stallard}
\address{Department of Mathematics and Statistics \\
The Open University \\
   Walton Hall\\
   Milton Keynes MK7 6AA\\
   UK}
\email{g.m.stallard@open.ac.uk}

\thanks{Both authors are supported by EPSRC grant EP/H006591/1}




\begin{abstract}
Let $f$ be a transcendental entire function and let $A(f)$ denote the set of points that escape to infinity `as fast as possible' under iteration. By writing $A(f)$ as a countable union of closed sets, called `levels' of $A(f)$, we obtain a new understanding of the structure of this set. For example, we show that if $U$ is a Fatou component in $A(f)$, then $\partial U\subset A(f)$ and this leads to significant new results and considerable improvements to existing results about $A(f)$. In particular, we study functions for which $A(f)$, and each of its levels, has the structure of an `infinite spider's web'. We show that there are many such functions and that they have a number of strong dynamical properties. This new structure provides an unexpected connection between a conjecture of Baker concerning the components of the Fatou set and a conjecture of Eremenko concerning the components of the escaping set.
\end{abstract}

\maketitle

\section{Introduction}
\setcounter{equation}{0} Let $f:\C\to\C$ be a transcendental entire function, and denote by
$f^{n},\,n=0,1,2,\ldots\,$, the $n$th iterate of~$f$. The {\it
Fatou set} $F(f)$ is defined to be the set of points $z \in \C$
such that $(f^{n})_{n \in \N}$ forms a normal
family in some neighborhood of $z$.  The complement of $F(f)$ is
called the {\it Julia set} $J(f)$ of $f$. An introduction to the
properties of these sets can be found in~\cite{wB93}.

This paper concerns the {\it escaping set} of $f$, defined as
follows:
\[
 I(f) = \{z: f^n(z) \to
 \infty \text{ as } n \to \infty \}.
\]
 For a
{\tef}~$f$, the escaping set was first studied  by
Eremenko~\cite{E} who proved that
\begin{equation}\label{E1}
I(f)\cap J(f)\ne \emptyset, \;
J(f)=\partial I(f),
\end{equation}
and that
\begin{equation}\label{E3}
\text{all the components of } \overline{I(f)} \text{ are unbounded}.
\end{equation}

Much of the work on $I(f)$ has been motivated by Eremenko's conjecture~\cite{E} that every component of $I(f)$ is unbounded. Significant progress was made on Eremenko's conjecture in~\cite{RS05} where it was shown that $I(f)$ always has at least one unbounded component. This result was proved by considering the {\it fast escaping set} $A(f)$, which was introduced by
Bergweiler and Hinkkanen in~\cite{BH99} and which now plays a key role in transcendental dynamics. In this paper we define
\[
A(f) = \{z: \mbox{there exists } L \in \N \text{ such that }
            |f^{n+L}(z)| \geq M^n(R,f), \text{ for } n \in \N\}.
\]
Here,
\[
M(r,f) = \max_{|z|=r} |f(z)|, \mbox{ for } r > 0,
\]
$M^n(r,f)$ denotes iteration of $M(r,f)$ with respect to the variable $r$, and $R>0$ can be taken to be any value such that $M(r,f) > r$ for $r\geq R$ or, equivalently, such that $M^n(R,f) \to \infty$ as $n \to \infty$. For simplicity, we will only write down this restriction on $R$ in formal statements of results -- elsewhere this should be assumed to be true. Note that $M(r,f) > r$ whenever $r > \min_{z \in
J(f)}|z|$;
otherwise, by Montel's thorem, we would have $\{z: |z| < r\} \subset F(f)$. We prove the basic properties of $A(f)$ in Section 2 -- these include the fact that the set $A(f)$ is independent of the choice of $R$.

Note that different definitions of $A(f)$ were used in~\cite{BH99}, \cite{RS05}, \cite{RS05a}, \cite{RS07}, \cite{RS08} and \cite{RS08c}. We show that all these definitions are equivalent in Section 2. The definition used here is of a similar form to the definition of the fast escaping set in a direct tract used in~\cite{BRS} and~\cite{RRS}.


We now introduce certain subsets of $A(f)$ based on
the above definition.

\begin{definition}
Let $f$ be a {\tef}, let $L \in \Z$ and let $R>0$ be such that $M(r,f) > r$ for $r\geq R$. The $L$th {\it level} of
$A(f)$ (with respect to $R$) is the set
\[
 A_R^L(f) = \{z:
            |f^n(z)| \geq M^{n+L}(R,f), \text{ for } n \in \N, n + L \geq 0 \}.
\]
We also put
\[
A_R(f) = A_R^0(f) = \{z:
            |f^n(z)| \geq M^n(R,f), \text{ for } n \in \N\}.
\]
\end{definition}

Note that, unlike the sets $I(f)$ and $A(f)$, each of the levels
of $A(f)$ is a {\it closed} set. Also, since $M^{n+1}(R,f) > M^n(R,f)$ for all
$n \geq 0$, we have $A_R^L(f) \subset A_R^{L-1}(f)$ for all $L \in
\Z$. This implies that $A(f)$ can be written as an expanding union
of closed sets:

\begin{equation}\label{A(f)def}
A(f) = \bigcup_{L \in \N} A_R^{-L}(f) \; \mbox{ and } \; A_R^{-L}(f)
\subset A_R^{-(L+1)}(f), \; L \in \N.
\end{equation}

In this paper we show that this concept of the levels of
$A(f)$ leads to significant new results concerning the
properties of $A(f)$. Further, it leads to considerable
improvements to many of the results in our
earlier papers~\cite{RS05}, \cite{RS05a} and \cite{RS07}. This paper is written in such a way that it provides a self-contained account of the main properties of $A(f)$, so some proofs of known results are included. In almost all cases, the proofs here are more straightforward than the original ones. 

In~\cite{RS05} we showed that all the components of $A(f)$ are unbounded. This is the strongest result for general entire functions that has been obtained towards Eremenko's
conjecture as, for any transcendental entire
function, $A(f) \neq \emptyset$ (see Theorem~\ref{E}) and so, as stated earlier, there is at least one unbounded component of $I(f)$.  Here we need the following more precise version of this result.

\begin{theorem}\label{main1}
Let $f$ be a {\tef} and let $R>0$ be such that $M(r,f) > r$ for $r\geq R$. Then, for each $L \in \Z$, each component of
$A_R^{L}(f)$ is closed and unbounded; in particular, each component of $A(f)$ is unbounded.
\end{theorem}

In Section 3 of this paper we give two proofs of Theorem~\ref{main1}. First we outline a constructive proof similar to the proof of~\cite[Theorem 1]{RS05}, the ideas of which have had significant applications (see~\cite{BFLM}, \cite{lR07} and~\cite{RRRS}). Then we give a new shorter proof by contradiction.

In Section 4, we consider Fatou components that meet $A(f)$. We prove the following, which is a key result of the paper and is used to deduce
many of the other results. Given its many applications, it has a
surprisingly simple proof which consists of combining the
new concept of the levels of $A(f)$ with known
distortion properties of the iterates of a function in a Fatou
component.

\begin{theorem}\label{main2}
Let $f$ be a {\tef}, let $R>0$ be such that $M(r,f) > r$ for $r\geq R$, and let $L \in \Z$. If $U$ is a Fatou component that
meets $A_R^{L}(f)$, then
\begin{itemize}
\item[(a)] $\overline{U} \subset A_R^{L-1}(f)$;
\item[(b)] if, in addition, $U$ is simply connected, then $\overline{U}
\subset A_R^{L}(f)$.
\end{itemize}
\end{theorem}

{\it Remark.} Note that Theorem~\ref{main2} implies that, if $U$ is a Fatou component
in $A(f)$, then the boundary of $U$ is also in $A(f)$. This
contrasts with the situation of a Fatou component in $I(f)$ which
may have boundary points that are not in $I(f)$. For
example, if $f(z) = z + 1 + e^{-z}$, then $F(f)$ is connected and
is in $I(f)$~\cite{pF26}. In fact, for this function, $F(f)$ is a Baker domain and so
is contained in the slow escaping set~\cite{RS08c}. Since the boundary of $F(f)$
is always equal to $J(f)$, the boundary of this Baker domain contains many points whose behaviour is
very different under iteration, such as fast escaping points and
periodic points.

We conclude Section 4 with an example which shows that it is not possible to replace $L-1$ by $L$ in Theorem~\ref{main2} part (a).

In Section 5, we consider the relationship between $A(f)$ and $J(f)$. The main result in this section is the following.

\begin{theorem}\label{main3}
Let $f$ be a {\tef}, let $R>0$ be such that $M(r,f) > r$ for $r\geq R$, and let $L \in \Z$. All the components of $A_R^{L}(f)
\cap J(f)$ are unbounded if and only if $f$ has no multiply connected Fatou components.
\end{theorem}

It follows immediately from Theorem~\ref{main3} that if $f$ has no multiply connected Fatou components, then {\it all} the components of $A(f) \cap J(f)$ are unbounded. In~\cite{RS05a} we used a relatively complicated
argument to prove the weaker result that if $f$ has  no multiply connected Fatou components,
then there is at least one unbounded component of $A(f)
\cap J(f)$. Here, Theorem~\ref{main3} follows from Theorem~\ref{main2} and Theorem~\ref{main1} in a straightforward manner, showing the value of introducing the concept
of the levels of $A(f)$.

In Sections 6--8, we consider the case when $A(f)$ has a structure which we now describe as an `infinite spider's web'. The first examples of functions for which $A(f)$ has this structure were given in~\cite[Theorem 2]{RS05}. It has become apparent that there are many transcendental entire functions for which $A(f)$ has this structure and that such functions have a number of strong dynamical properties. Because of the increasing significance of this structure, we introduce the following terminology.

\begin{definition}
A set $E$ is an {\it (infinite) spider's web} if $E$ is connected and there
exists a sequence of bounded simply connected domains $G_n$ with $G_n \subset G_{n+1}$, for $n \in \N$,
$\partial G_n \subset E$, for $n \in \N$ and $\bigcup_{n \in \N}G_n = \C$.
\end{definition}

{\it Remark.} If $I(f)$ is a spider's web, then $I(f)$ is connected (and unbounded) and so Eremenko's conjecture holds.\\

In~\cite[Theorem 2]{RS05} we showed that, if $f$ is a {\tef}, then
$A(f)$ and $I(f)$ are spiders' webs whenever $f$ has a multiply connected Fatou component. In fact $A_R(f)$ is also a spider's web for such a function. This is a corollary of the following more general result, proved in Section~6. A version of Theorem~\ref{main4} was proved in~\cite{RS07}.

\begin{theorem}\label{main4}
Let $f$ be a {\tef} and let $R>0$ be such that $M(r,f) > r$ for $r\geq R$. If $A_R(f)^c$ has a bounded component, then each of $A_R(f)$, $A(f)$ and $I(f)$ is a {\sw}.
\end{theorem}

It follows from Theorem~\ref{main4} that if $A_R(f)$ is a {\sw}, then so are $A(f)$ and $I(f)$. We now know that there are many functions for which
$A_R(f)$ is a spider's web; see Theorem~\ref{main8}. In fact all of the functions for which we know that $I(f)$ is a spider's web have the stronger property that $A_R(f)$ is a spider's web.

An approximation to $A_R(f)$ for the function
\[
f(z) = \frac{1}{2} \left( \cos z^{1/4} + \cosh z^{1/4}\right) = 1 + \frac{z}{4!} + \frac{z^2}{8!} + \cdots
\]
is shown in Figure 1. (This function was mentioned in \cite[Section 6]{RS07}.) This structure of $A_R(f)$ is novel, and clearly very
different to the familiar structure of a Cantor bouquet, which arises when describing the escaping sets of many entire functions.

In this diagram, the set of black points is an approximation to $A_R(f)$, for some $R>0$. The large grey region on the right is the immediate basin of an attracting fixed point, and the fine structure of $A_R(f)$ around this basin extends to the rest of the plane, though the computer cannot handle the arithmetic at most of these points. The range shown is given by $|\re z|, |\im z| \leq 2 \times 10^4$.

\begin{figure}[htb]
\begin{center}
\includegraphics[width=8cm]{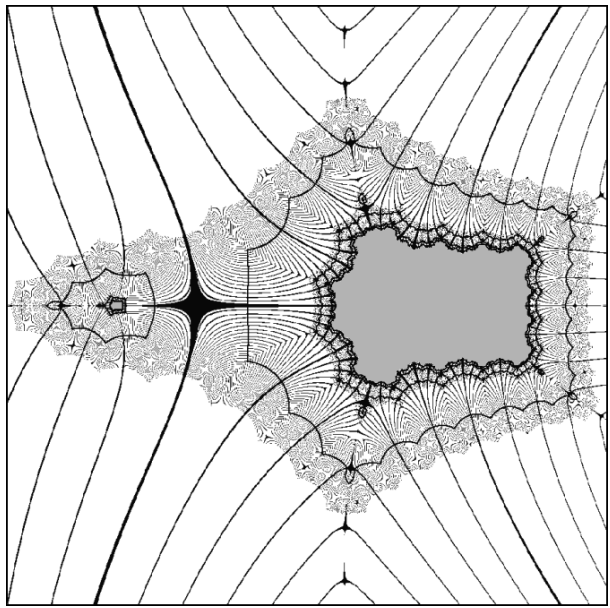}
\caption{An $A_R(f)$ spider's web}
\end{center}
\end{figure}

In Section 7, we show that functions for which $A_R(f)$ is a
spider's web have many strong dynamical properties. We begin by proving the
following result which follows easily from Theorem~\ref{main2}.
Part (a) is new and a version of part (b) was proved in~\cite{RS07}, though the
proof given here is more straightforward.

\begin{theorem}\label{main6}
Let $f$ be a {\tef}, let $R>0$ be such that $M(r,f) > r$ for $r\geq R$, and let $A_R(f)$ be a spider's web.
\begin{itemize}
\item[(a)] If $f$ has no multiply connected Fatou components, then each of
\[
A_R(f) \cap J(f), \; A(f) \cap J(f), \; I(f) \cap J(f) \mbox{ and } J(f)
\]
is a
spider's web.
\item[(b)] The function $f$ has no unbounded Fatou components.
\end{itemize}
\end{theorem}

Part (b) of Theorem~\ref{main6} allows us to make a connection between the existence of an $A_R(f)$ spider's web and a conjecture of Baker about functions of small growth; see Section 8.

We also use the concept of the levels of $A(f)$ to prove the
following result about the complement of $A(f)$.

\begin{theorem}\label{holes}
Let $f$ be a {\tef}, let $R>0$ be such that $M(r,f) > r$ for $r\geq R$, and let $A_R(f)$ be a spider's web.
\begin{itemize}
\item[(a)] All the components of $A(f)^c$ are compact.
\item[(b)] Every point of $J(f)$ is the limit of a sequence of points, each of which lies in a distinct component of $A(f)^c$.
\end{itemize}
\end{theorem}

The property of $A(f)$ proved in part (a) of this result contrasts with the fact that all the components of $A_R^L(f)^c$, $L \in \Z$, are open. This result demonstrates that, if $A_R(f)$ is a spider's web, then $A(f)$ has a very intricate structure. Further results about the intricate nature of $A(f)$ when $A_R(f)$ is a spider's web have been obtained by Osborne~\cite{Os}; for example, the result in part~(b) can be strengthened to state that singleton components of $A(f)^c$ are dense in $J(f)$.

We end Section 7 by proving the following result whose proof also uses the concept of the levels of $A(f)$.

\begin{theorem}\label{main7}
Let $f$ be a {\tef}, let $R>0$ be such that $M(r,f) > r$ for $r\geq R$, and let $A_R(f)$ be a spider's web.
\begin{itemize}
\item[(a)]
Each point in $I(f)$
belongs to an unbounded continuum in $I(f)$ on which all points
escape to infinity uniformly.

\item[(b)] If $K$ is a component of $A(f)^c$, then either $K \cap
I(f) = \emptyset$ or all points in $K$ escape to infinity uniformly.
\end{itemize}
\end{theorem}

{\it Remark.} Theorem~\ref{main7} part (a) answers a question asked by Rempe in~\cite{lR07}. It is shown in~\cite{RRRS} that the conclusion of this part holds for many functions in the Eremenko-Lyubich class $\B$ (defined in Section 3). Rempe~\cite{lR09} has shown the existence of a {\tef} in the class $\B$ for which the conclusion of Theorem~\ref{main7} part (a) does not hold.\\

These results show that, if $A_R(f)$ is a spider's web, then
$f$ has many interesting dynamical properties. In Section 8, we investigate which families of functions have these properties. Previously it was
thought that it is relatively unusual for the escaping set to have this structure. Indeed, the following result shows that many commonly studied entire functions do {\it not} have an $A_R(f)$ spider's web.

\begin{theorem}\label{SW2}
Let $f$ be a {\tef}, let $R>0$ be such that $M(r,f) > r$ for $r\geq R$, and let $A_R(f)$ be a {\sw}. Then there is no path to infinity on which $f$ is bounded and so
\begin{itemize}
\item[(a)] $f$ does not belong to the class $\B$;
\item[(b)] $f$ has no exceptional points (that is, points with a finite backwards orbit).
\end{itemize}
\end{theorem}

However, we show that there are in fact many functions for which $A_R(f)$ is a spider's web. Thus, for all these functions, the sets $A_R(f)$, $A(f)$ and $I(f)$ are connected, and $f$ has no unbounded Fatou components, and so both Eremenko's conjecture and Baker's conjecture hold.

\begin{theorem}\label{main8}
Let $f$ be a {\tef} and let $R>0$ be such that $M(r,f) > r$ for $r\geq R$. Then $A_R(f)$ is a spider's web if one of the
following holds:
\begin{itemize}
\item[(a)] $f$ has a multiply connected Fatou component;
\item[(b)] $f$ has very small growth; that is,
 there exist $m \geq 2$ and $r_0>0$ such that
\begin{equation}
   \log \log M(r,f) < \frac{\log r}{ \log^m r},
    \; \mbox{ for } r>r_0,
\end{equation}
where $\log^m$ is the $m$th iterated logarithm;
\item[(c)] $f$ has order less than $1/2$ and regular growth;
\item[(d)] $f$ has finite order, Fabry gaps and regular growth;
\item[(e)] $f$ has a sufficiently strong version of the pits effect and has regular growth.
\end{itemize}
\end{theorem}

Parts (a) to (c) follow from results in~\cite{RS07}
and~\cite{RS08}. Parts (d) and (e) are new -- there are many
functions which belong to these classes. (We will define Fabry gaps, regular growth and the pits effect
in Section 8.)

We conclude Section 8 by showing that $A_R(f)$ is a spider's web if and only if $A_R(f^m)$ is a spider's web. This enables us to construct many new examples of functions for which $A_R(f)$ is a spider's web, by considering iterates of the functions described in Theorem~\ref{main8}. In particular, this makes it easy to construct examples of such functions for which the order is infinite.

{\it Remark.} Finally, many of the new results about $A(f)$ obtained here can be generalised in some form to transcendental meromorphic functions with direct tracts; see~\cite{BRS} for results concerning the fast escaping set in a direct tract.

\section{Basic properties of $A(f)$}
\setcounter{equation}{0}
In this section we prove the basic properties of the fast escaping set and also show that there are several equivalent definitions of this set.

 Let $f$ be a {\tef} and let $r>0$. Recall that the fast escaping set $A(f)$ is defined in terms of the maximum modulus function $M(r,f)$. For simplicity, from now on we will write this as $M(r)$ provided that it is clear from the context which function $f$ is being considered. We often use the following facts about $M(r)$ without comment:

\begin{equation}\label{M1}
   \mbox{ if } |z| < M^m(r), \mbox{ for some } m \in \N, \mbox{ then } |f^n(z)| < M^{n+m}(r), \mbox{ for } n \in \N;
   \end{equation}
   \begin{equation}\label{M2}
 \log M(r) / \log r \to \infty \mbox{ as } r \to \infty.
 \end{equation}

 Many results in this paper use the following property of the maximum modulus function, which was proved in this form in~\cite[Lemma 2.2]{RS08}.

\begin{lemma}\label{convex}
Let $f$ be a {\tef}. Then there exists $R>0$ such that, for all $r \geq R$ and all $c>1$,
\[
 \log M(r^c) \geq c\log M(r).
\]
\end{lemma}

It follows from Lemma~\ref{convex} together with~\eqref{M2} that, if $k>1$, then
\begin{equation}\label{con}
\frac{M(kr)}{M(r)} \to \infty \mbox{ as } r \to \infty.
\end{equation}

Now let $R>0$ be such that $M(r,f) > r$ for $r\geq R$. Then
 \begin{equation}\label{MR}
   M^n(R) \to \infty \mbox{ as } n \to \infty.
 \end{equation}
 Recall that
\[
 A_R(f) = \{z: |f^n(z)| \geq  M^n(R), \mbox{ for } n \in \N\}.
\]
Also, for each $L \in \Z$,
\[
A_R^L(f) = \{z: |f^{n}(z)| \geq M^{n+L}(R), \mbox{ for } n \in \N, n+L \geq 0\}
\]
and the fast escaping set is
\begin{equation}\label{A(f)}
 A(f) = \bigcup_{L \in \N} A_R^{-L}(f).
\end{equation}

The following set relations follow immediately from the definition of $A_R^L(f)$ and are often used without comment:

\begin{equation}\label{A1}
A_R^L(f) \subset \{z: |z| \geq M^L(R)\}, \mbox{ for } L \geq 0;
\end{equation}

\begin{equation}\label{subset}
f(A_R^L(f)) \subset A_R^{L+1}(f) \subset A_R^L(f), \mbox{ for } L \in \Z.
\end{equation}

We now give the following elementary result concerning $A(f)$. Parts (a)--(c) were stated in~\cite{BH99} and part (d) was proved in~\cite{RS05}.

  \begin{theorem}\label{basic}
Let $f$ be a {\tef}, let $R>0$ be such that $M(r,f) > r$ for $r\geq R$, and let $A(f) = \bigcup_{L \in \N} A_R^{-L}(f)$. Then
\begin{itemize}
\item[(a)] $A(f)$ is completely invariant under $f$;
\item[(b)] $A(f)$ is independent of $R$;
\item[(c)] $A(f) \subset Z(f) = \{z: \frac{1}{n} \log \log |f^n(z)| \to \infty \mbox{ as } n \to \infty\}$;
\item[(d)] if $g = h^{-1}fh$, where $h(z) = az+b$, $a \neq 0$, then $A(f) = h(A(g))$.
\end{itemize}
\end{theorem}

{\it Remark.} The set $Z(f)$, which consists of the points that `zip' to infinity, was studied in~\cite{RS00}. (Some authors denote this set by $T(f)$.)

\begin{proof}

{\it (a)} The complete invariance of $A(f)$ under $f$ follows directly from~\eqref{A(f)} and~\eqref{subset}.

{\it (b)} We take $R'>R$. Clearly $A_{R'}^L(f) \subset A_R^L(f)$, for $L \in \Z$, and so
\[
  \bigcup_{L \in \N} A_{R'}^{-L}(f) \subset \bigcup_{L \in \N} A_R^{-L}(f).
\]
Now note that, by~\eqref{MR}, there exists $m \in \N$ such that $M^m(R) > R'$ and so
\[
\bigcup_{L \in \N} A_{R'}^{-L}(f) \supset \bigcup_{L \in \N} A_{M^m(R)}^{-L}(f)
= \bigcup_{L \in \N} A_R^{m-L}(f) \supset \bigcup_{L \in \N} A_R^{-L}(f).
\]
Together with~\eqref{A(f)}, these set relations show that $A(f)$ is independent of $R$.

 {\it (c)} To prove that $A(f) \subset Z(f)$, we use the fact that, since $f$ is a {\tef}, it follows from~\eqref{M2} and~\eqref{MR} that $\log M^{n+1}(R) / \log M^n(R) \to \infty$ as $n \to \infty$. Thus, for each $C>e$, there exists $N \in \N$ such that $\log M^{n+1}(R) > C\log M^n(R)$ for all $n \geq N$ and $\log M^N(R) \geq 1$. So, for $n > 2(N+L)$, $L \in \N$,
\[
\log M^{n-L}(R) > C^{n/2},
\]
and hence, if $z \in A_R^{-L}$, then
\[
\frac{1}{n} \log \log |f^n(z)| \geq \frac{1}{n} \log \log M^{n-L}(R) > \frac{1}{2} \log C.
\]
The result now follows since we can choose $C$ to be arbitrarily large.

{\it (d)} This property follows immediately from the equivalent definition of $A(f)$ as the set $B(f)$ given in~\cite{RS05}; see Corollary~\ref{equal}.
\end{proof}

It was shown in~\cite{BH99} that Eremenko's construction in~\cite{E} of points in $I(f)$  actually gives points that are in $A(f)$. In fact his construction, which is based on Wiman-Valiron theory, can readily be adapted to give points that are in $A(f)$ and also have other strong dynamical properties. We summarise the main properties of such points in the following result. Points $z'$ with these properties are particularly useful and we often refer to them as {\it Eremenko points}.

\begin{theorem}\label{E}
Let $f$ be a {\tef} and let $\eps \in (0,1)$. There exists $R>0$ such that, if $r>R$, then there exists
  \[
  z' \in \{z:r \leq |z| \leq r(1 + \eps)\} \cap A(f)
   \]
   with
   \[
    |f^n(z')| > M^n(r), \mbox{ for } n \in \N.
   \]
Further, we can choose $z'$ such that there exist sequences $(z_n)_{n \geq 0}$ and $(k_n)_{n \geq 0}$ and sets
\[
D_n = \{z: |z - z_n| < \frac{1}{4}\eps |z_n|\} \mbox{ and } A_{n+1} = \{z: \frac{1}{k_n} M(|z_n|) \leq |z| \leq k_nM(|z_n|)\}
\]
such that

\begin{itemize}
\item[(a)]
for $n \geq 0$,
\[
f^n(z') \in \overline{D_n}, \; f(D_n) \supset A_{n+1} \mbox{ and } |f(z_n)| = M(|z_n|)\geq M^{n+1}(|z_0|);
\]

\item[(b)] for $n \in \N$,
\[
 D_n \subset A_n \cap \{z: |z| \geq M(|z_{n-1}|)\};
\]

\item[(c)]$\{z'\} = \bigcap_{n = 0}^{\infty} B_n$, where $B_0 = \overline{D_0}$ and, for $n \in \N$, $B_n$ is a component of $f^{-n}(\overline{D_n})$;

\item[(d)] $(k_n)$ is a positive increasing sequence with $k_n \to \infty$ as $n \to \infty$.
\end{itemize}
\end{theorem}

{\it Remark.} By modifying the definitions of the discs and annuli in Theorem~\ref{E} appropriately, each disc $D_n$ can be chosen to lie in a thin annulus $A_n'$ contained in and concentric with $A_n$. For any $\delta > 0$, the modulus of the annulus $A_n'$ can be chosen to be less than $\delta$ for $n$ sufficiently large. Note that $D_n$ can be allowed to lie in $\{z: |z| \leq M(|z_{n-1}|)\}$ but, with this modification, the point $z'$ does not necessarily satisfy $|f^n(z')| > M^n(r)$, for $n \in \N$.\\

     We now show how Eremenko points can be used to demonstrate that the definitions of the fast escaping set used in earlier papers are equivalent to the definition used in this paper. The original definition of the fast escaping set was given by Bergweiler and Hinkkanen in~\cite{BH99}. For clarity, we refer to the set given by their definition as $A'(f)$ until we have proved that it is equal to the set $A(f)$ defined earlier:
\[
 A'(f) = \{z: \mbox{there exists } L \in \N \text{ such that }
          |f^{n+L}(z)| > M(R,f^n), \text{ for } n \in \N\},
\]
where $R$ is any value such that $R>\min_{z \in J(f)}|z|$.

In~\cite{RS05} we showed that $A'(f)$ is equal to the set
\[
  B(f) = \{z: \mbox{there exists } L \in \N \text{ such that }
            f^{n+L}(z) \notin \widetilde{f^n(D)}, \text{ for } n \in
            \N\},
\]
where $D$ is any open disc meeting $J(f)$, and $\widetilde{U}$
denotes the union of $U$ and its bounded complementary
components.

We now show that all of these definitions of the fast escaping set are equivalent. In order to do this, we prove the following result.

\begin{lemma}\label{subsets}
Let $f$ be a {\tef} and let $D = \{z: |z| < R\}$. If $R>0$ is sufficiently large, then
\begin{eqnarray*}
\{z:|z| \leq M^n(R/2,f)\} & \subset & \widetilde{f^n(D)} \subset \{z: |z| \leq M(R,f^n)\}\\
 & \subset & \{z:|z| \leq M^n(R,f)\}\\
 & \subset & \{z: |z| < M^{n+1}(R,f)\}.
\end{eqnarray*}
\end{lemma}
\begin{proof}
The only set relation that is not clear in the above list is the first one. To prove this, we note that, by Theorem~\ref{E}, if $R$ is sufficiently large, then $\{z: R/2 < |z| < R\}$ contains a neighbourhood of an Eremenko point. The result now follows from the properties of such a point described in parts (a) and (c) of Theorem~\ref{E}.
\end{proof}

The equivalence of the various definitions of the fast escaping set follows immediately from Lemma~\ref{subsets}. Note that the last set relation in Lemma~\ref{subsets} is required since the definition of $A'(f)$ involves a strong inequality.

\begin{corollary}\label{equal}
Let $f$ be a {\tef}. Then $A(f) = A'(f) = B(f)$.
\end{corollary}

{\it Remarks.} 1. The proof of Corollary~\ref{equal} is much simpler than the proof that $B(f) = A'(f)$ given in~\cite{RS05}.

2. In~\cite{RS07} we based the proofs of several results about $A(f)$ and $I(f)$ on the set
\[
  B_D(f) = \{ z: f^n(z)\notin \widetilde{f^n(D)}, \mbox{ for } n \in \N \},
\]
where $D$ is any open disc meeting $J(f)$. This set is analogous
to the set $A_R(f)$ but we do not have $f(B_D(f)) \subset
B_D(f)$ (compare with~\eqref{subset}) and so more complicated proofs are required when working
with $B_D(f)$ instead of $A_R(f)$.\\

We now use Lemma~\ref{subsets} to prove the following result.

\begin{theorem}\label{eq}
Let $f$ be a {\tef} and let $m \in \N$. If $R>0$ is sufficiently large, then
\[
 A_R(f) \subset A_R(f^m) \subset A_{R/2}(f)
\]
and hence $A(f) = A(f^m)$.
\end{theorem}
\begin{proof}
Let $R > 0$ be such that $M(r,f) > r$ for $r\geq R$, and let $m \in \N$. If $z \in A_R(f)$, then
\[
|f^n(z)| \geq M^n(R,f), \; \mbox{ for } n \in \N,
\]
 so
\[
 |f^{nm}(z)| \geq M^{nm}(R,f) \geq M^n(R,f^m), \; \mbox{ for } n \in \N,
\]
and hence $z \in A_R(f^m)$. Thus $A_R(f) \subset A_R(f^m)$ and hence $A(f) \subset A(f^m)$.

Conversely, suppose that $R>0$ is sufficiently large. Then, by Lemma~\ref{subsets},
\begin{equation}\label{r0}
  M(R,f^n) \geq M^n(R/2,f), \mbox{ for } n \in \N.
\end{equation}
If $z \in A_R(f^m)$, then
\[
|f^{nm}(z)| \geq M^n(R,f^m), \; \mbox{ for } n \in \N,
\]
so, by~\eqref{r0},
\[
 |f^{nm}(z)| \geq M(R,f^{nm}) \geq M^{nm}(R/2,f), \; \mbox{ for } n \in \N.
\]
This implies that $|f^n(z)| \geq M^n(R/2,f)$ for all $n \in \N$ and hence $z \in A_{R/2}(f)$.
Thus $A_R(f^m) \subset A_{R/2}(f)$ and hence $A(f^m) \subset A(f)$.
\end{proof}

We conclude this section by giving a condition that is sufficient to ensure that certain points are in $A(f)$. In forthcoming work~\cite{RSreg} we give a weaker condition that is sufficient to ensure that points are in $A(f)$ provided that $f$ has sufficiently regular growth.

\begin{theorem}\label{Q(f)}
Let $f$ be a {\tef} and, for $\eps \in (0,1)$, $r>0$, let
 $\mu(r) = \eps M(r)$.
 Then
\[
A(f) = \{z: \mbox{there exists } L \in \N \text{ such that }
            |f^{n+L}(z)| \geq \mu^n(R), \text{ for } n \in \N\},
\]
where $R>0$ is sufficiently large to ensure that $\mu(r) > r$, for $r \geq R$.
\end{theorem}
\begin{proof}
By hypothesis,
\begin{equation}\label{mu}
\mu^n(r) \to \infty \mbox{ as } n \to \infty, \mbox{ for } r \geq R.
\end{equation}
(The existence of such an $R$ follows from~\eqref{M2}.)

Further, by~\eqref{con}, there exists $R' \geq R$ such that, for $r \geq R'$,
\begin{equation}\label{mun}
 \mu(r) = \eps M(r) \geq \frac{1}{\eps} M(\eps r) \geq r \mbox{ and hence }
 \mu^n(r) \geq M^n(\eps r), \mbox{ for } n \in \N.
\end{equation}
It follows from~\eqref{mu} that there exists $M \in \N$ such that $\mu^M(R) \geq R'/\eps$. So, by~\eqref{mun},
\[
 \mu^{n+M}(R) \geq \mu^n(R'/\eps) \geq M^n(R'), \mbox{ for } n \in \N.
\]
So, if there exists $L \in \N$ such that $|f^{n+L}(z)| \geq \mu^n(R)$, for $n \in \N$, then
\[
 |f^{n+M+L}(z)| \geq \mu^{n+M}(R)\geq M^n(R') \geq M^n(R), \mbox{ for } n \in \N,
\]
and hence $z \in A(f)$.
\end{proof}

This result can be applied, for example, to the function $f(z) = \sinh z + z^2$ to show that, for $R$ sufficiently large, $(-\infty, -R] \subset A(f)$, since $f((-\infty, -R]) \subset (-\infty, -R]$ and
\[
|f(-r)| \geq \frac{1}{2} f(r) = \frac{1}{2} M(r), \mbox{ for } r > R.
\]

\section{The components of $A(f)$}
\setcounter{equation}{0}

 In this section we show that each component of each of the levels of $A(f)$ is unbounded and hence all the components of $A(f)$ are unbounded. This is sufficient to prove Theorem~\ref{main1} since each of the levels is a closed set and hence each component of each of the levels must also be closed. We give two proofs of Theorem~\ref{main1}. The first is an outline of the constructive proof that all the components of $A(f)$ are unbounded that we gave in~\cite{RS05} and the second is a new proof by contradiction.

\begin{proof}[First proof of Theorem~\ref{main1}]
Let $z_0 \in A_R^L(f)$, for some $L \in \Z$. Then, for all $n \in \N$ with $n+L \geq 0$,
\[
 f^n(z_0) \in \{z: |z| \geq M^{n+L}(R)\} = E_n,
\]
say. Now let $L_n$ denote the component of $f^{-n}(E_n)$ that contains $z_0$. Then $L_n$ is closed and also unbounded, since $f^n$ is entire. Moreover,
\[
L_{n+1} \subset L_n, \; \mbox{ for } n \in \N, n+L \geq 0.
\]
Indeed $f^{n+1}(z) \in E_{n+1}$ implies that $f^n(z) \in E_n$ so $L_{n+1} \subset f^{-n}(E_n)$. Hence
\[
 K = \bigcap_{n \in \N, \; n+L \geq 0} (L_n \cup \{\infty\})
\]
is a closed connected subset of $\hat{\C}$ which contains $z_0$ and $\infty$. Now let $\Gamma$ be the component of $K \setminus \{\infty\}$ which contains $z_0$.  Then $\Gamma$ is closed in $\C$ and unbounded; see~\cite[page 84]{New}.
Finally, we note that $\Gamma \subset A_R^L(f)$ since, if $z \in \Gamma$, then  $f^n(z) \in E_n$ and so
\[
 |f^n(z)| \geq M^{n+L}(R), \; \mbox{ for } n \in \N, \; n+L \geq 0.
 \qedhere
\]
\end{proof}

\begin{proof}[Second proof of Theorem~\ref{main1}]
Suppose that there exists a bounded component $K$ of $A_R^L(f)$, for some $L \in \Z$. Then, since $A_R^L(f)$ is closed, there exists a Jordan curve $\gamma \subset A_R^L(f)^c$ that surrounds $K$; see~\cite[page 143]{New}. For each $n \in \N$ with $n + L \geq 0$, we put
\[
\gamma_n = \{z \in \gamma: |f^n(z)| \geq M^{n+L}(R)\}.
\]
Since $K \subset A_R^L(f)$, it follows from the maximum principle that $\gamma_n \neq \emptyset$. Further, for $n \in \N$ with $n + L \geq 0$, $\gamma_n$ is closed and $\gamma_{n+1} \subset \gamma_n$. Thus
\[
\bigcap_{n \in \N, \;n+L \geq 0} \gamma_n \neq \emptyset.
\]
This, however, is a contradiction since
\[
\bigcap_{n \in \N,\; n+L \geq 0} \gamma_n \subset A_R^L(f) \cap \gamma \; \mbox{ and } \;
\gamma \subset A_R^L(f)^c.
\qedhere
\]
\end{proof}

\section{Fatou components and $A(f)$}
\setcounter{equation}{0}

 In this section we consider which types of Fatou components can meet $A(f)$ and prove various results about such components. In particular, we prove Theorem~\ref{main2}. We begin this section with a distortion lemma that we use to prove various properties of $A(f)$. The original proof of this lemma is due to Baker~\cite[Lemmas 1 and~2]{iB88} and a good account of the result is also given in~\cite[Lemma 7]{wB93}.

\begin{lemma}\label{dist}
Let $f$ be a {\tef}, let $U \subset I(f)$ be a Fatou component and let $K$ be a compact subset of $U$.
\begin{itemize}
\item[(a)] There exist $C>1$, $N \in \N$ such that
\[
|f^n(z_0)| \leq |f^n(z_1)|^C, \; \mbox{ for } z_0, z_1 \in K, \; n \geq N.
\]
\item[(b)] If, in addition, $U$ is simply connected, then there exist $C>1$, $N \in \N$ such that
\[
|f^n(z_0)| \leq C|f^n(z_1)|, \; \mbox{ for } z_0, z_1 \in K, \; n \geq N.
\]
\end{itemize}
\end{lemma}

The following corollary of Lemma~\ref{dist} was proved in~\cite{BH99}.

\begin{corollary}
Let $f$ be a {\tef} and let $U$ be a Fatou component with $U \cap A(f) \neq \emptyset$. Then $U$ is a wandering domain.
\end{corollary}
\begin{proof}
 Let $U \subset I(f)$ be a periodic Fatou component of period $p$, say, and let $z_0 \in U$. It follows from Lemma~\ref{dist} part (a) applied to the entire function $f^p$ with $K = \{z_0, f^p(z_0)\}$ that there exist $C>1$, $N \in \N$ such that, for each $n \geq N$, $\log |f^{np}(z_0)| \leq C^n$ and hence
$z_0 \notin Z(f)$. Thus, by Theorem~\ref{basic} part (c), $z_0 \notin A(f)$.
\end{proof}

There are various types of wandering domains that are known to lie in $A(f)$. In particular, the following result due to Baker~\cite[Theorem 3.1]{iB84} can be used to show that $U$ belongs to $A(f)$ if $U$ is multiply connected.

\begin{lemma}\label{Baker}
Let $f$ be a {\tef} and let $U$ be a multiply connected Fatou component. Then
\begin{itemize}
\item each $f^n(U)$ is bounded,
\item $f^{n+1}(U)$ surrounds $f^n(U)$ for large $n$,
\item $f^n(U) \to \infty$ as $n\to\infty$.
 \end{itemize}
\end{lemma}

{\it Remark.} Note that if $U$ is a bounded Fatou component of a {\tef} $f$ then, for $n \in \N$, $f^n(U)$ is also a Fatou component of $f$; see~\cite{mH98}.

We showed in~\cite{RS05} that, for any multiply connected Fatou component $U$ we have $\overline{U} \subset A(f)$. The following more precise result holds.

\begin{theorem}\label{Baker1}
Let $f$ be a transcendental entire function and let $U$ be a multiply connected Fatou component. Then, for some $N \in \N$ and $R>0$, we have $f^N(\overline{U}) \subset A_R(f)$ and hence $\overline{U} \subset A(f)$.
\end{theorem}
\begin{proof}
 Suppose that $D = \{z: |z| < R\}$, where $R$ is sufficiently large to apply Lemma~\ref{subsets},  and that $U$ is a multiply connected Fatou component. Then, by Lemma~\ref{Baker}, there exists $N \in \N$ such that $\widetilde{f^N(U)} \supset D$ and
 \begin{equation}\label{sur}
 f^{n+1}(U) \mbox{ surrounds } f^n(U), \mbox{ for } n \geq N.
 \end{equation}
 Then
\begin{equation}\label{D}
f^n(D) \subset f^n(\widetilde{f^N(U)}) \subset \widetilde{f^{n+N}(U)}, \; \mbox{ for } n \in \N,
\end{equation}
so, by~\eqref{sur} and~\eqref{D},
\[
   \overline{f^{n+N+1}(U)} \cap \widetilde{f^n(D)} = \emptyset, \mbox{ for } n \in \N.
 \]
 Hence, by Lemma~\ref{subsets},
 \[
  \overline{f^{n+N+1}(U)} \cap \{z: |z| < M^n(R/2)\} = \emptyset, \mbox{ for } n \in \N,
 \]
 so $\overline{f^{N+1}(U)} \subset A_{R/2}(f)$ as required.
\end{proof}

Alternatively, Theorem~\ref{Baker1} can be proved by using the following result due to Zheng~\cite{jhZ06}.
\begin{lemma}\label{Zheng}
Let $f$ be a {\tef}. If $f$ has a multiply connected Fatou component $U$, then there exist annuli $A_n =\{z:r_n <
|z|<R_n\}$, $n\in\N$, and $n_0\in\N$ such that
\[A_n\subset f^n(U), \quad\text{for } n > n_0,\]
and $R_n/r_n \to\infty$ as
$n\to\infty$.
\end{lemma}

{\it Remarks.} 1. Zheng proved a version of this result for the larger class of transcendental meromorphic functions with at most finitely many poles, where $U$ is a wandering domain with the properties listed in Lemma~\ref{Baker}.

2. In~\cite{BRS2} (see also~\cite{SO}) we show that Lemma~\ref{Zheng} can be strengthened to give much larger annuli inside multiply connected Fatou components.\\

It follows from Lemma~\ref{Zheng} and Theorem~\ref{E} that, if $U$ is a multiply connected Fatou component of a {\tef} $f$ and $n$ is sufficiently large, then $f^n(U)$ contains an Eremenko point. Together with Theorem~\ref{E} part (a) and Lemma~\ref{Baker}, this is sufficient to prove Theorem~\ref{Baker1}.

There are also examples of simply connected Fatou components in $A(f)$. Bergweiler~\cite{wB08} has constructed an example of a {\tef} with
simply connected bounded wandering domains that lie between multiply connected
wandering domains, and whose closures are in $A(f)$. These appear to be the only
known examples of Fatou components in $A(f)$ that are not
multiply connected. This suggests the following question.\\

{\bf Question 1} Can there be any unbounded Fatou components in
$A(f)$?\\

For the examples of Fatou components that meet $A(f)$ described above, it is the case that the closure of the Fatou component is also in $A(f)$. We now show that this property is true in general by proving Theorem~\ref{main2}.

\begin{proof}[Proof of Theorem~\ref{main2}.]
 Let $U$ be a Fatou component and let
 $z_0 \in U \cap A_R^L(f)$, where $L\in \Z$. Then $U \subset I(f)$, by normality, and
\begin{equation}\label{U1}
|f^n(z_0)| \geq M^{n+L}(R), \; \mbox{ for } n \in \N \mbox{ with } n+L \geq 0.
\end{equation}
 {\it (a)} To prove that $\overline{U} \subset A_R^{L-1}(f)$, we suppose that $z_1 \in U$. It follows from Lemma~\ref{dist} part (a), \eqref{U1} and properties of the maximum modulus function that there exist $C>1$ and $N \in \N$ such that
\[
  |f^{n}(z_1)| \geq |f^{n}(z_0)|^{1/C} \geq (M^{n+L}(R))^{1/C} \geq M^{n+L-1}(R), \; \mbox{ for } n \geq N.
\]
Thus $U \subset A_R^{L-1}(f)$. Since $A_R^{L-1}(f)$ is a closed set, this implies that $\overline{U} \subset A_R^{L-1}(f)$ as claimed.

{\it (b)} Now suppose that $U$ is simply connected. To prove that $\overline{U} \subset A_R^L(f)$, we use proof by contradiction. Suppose that there exists $z_1 \in U$ with $z_1 \notin A_R^L(f)$. Then there exist $N \in \N$ and $c>1$ such that
\[
|f^{N}(z_1)| = (M^{N+L}(R))^{1/c} = R_0,
\]
say. Since $z_1 \in I(f)$, we can assume that $M^n(R_0) \to \infty$ as $n \to \infty$.
Together with \eqref{U1}, this implies that, for $n \in \N$,
\[
 \frac{|f^{n+N}(z_0)|}{|f^{n+N}(z_1)|} \geq \frac{M^{n+N+L}(R)}{M^n((M^{N+L}(R))^{1/c})} =
 \frac{M^n(R_0^c)}{M^n(R_0)}.
\]

This, however, contradicts Lemma~\ref{dist} part (b) since, by
Lemma~\ref{convex},
\[
\frac{M^n(R_0^c)}{M^n(R_0)} \geq \frac{(M^n(R_0))^c}{M^n(R_0)} = (M^n(R_0))^{c-1}
 \to \infty \; \mbox{ as } n \to \infty.
 \]
Thus $U \subset A_R^L(f)$. Since $A_R^L(f)$ is a closed set, this implies that $\overline{U} \subset A_R^L(f)$.
\end{proof}

The next result follows immediately from Theorem~\ref{main2} part (b), and will be used later in the paper.

\begin{corollary}\label{AJF}
Let $f$ be a {\tef}, let $R>0$ be such that $M(r,f) > r$ for $r\geq R$, and let $L \in \Z$. If all the components of $F(f)$ are simply connected, then $\partial A_R^L(f) \subset J(f)$.
\end{corollary}
{\it Remark.}  It follows from Montel's theorem that, if $f$ is a {\tef} and $L \in \Z$, then all the interior points of $A_R^L(f)$ are in $F(f)$.\\

We conclude this section by showing that the conclusion of Corollary~\ref{AJF} can fail if there is a multiply connected Fatou component. Suppose that $f$ is a {\tef} and that $U$ is a multiply connected component of $F(f)$. Then, by Lemma~\ref{Zheng} and Theorem~\ref{E}, there exist $N \in \N$ and $R>0$ such that
\[
 f^N(U) \supset \{z: \frac{1}{2}R \leq |z| \leq 2R\}
\]
and such that $f^N(U)$ contains an Eremenko point $z_0$ with $|f^n(z_0)| \geq M^n(R)$, for $n \in \N$. Since
\[
z_0 \in A_R(f) \cap f^N(U)
\]
and, by~\eqref{A1},
\[
 \{z: \frac{1}{2}R \leq |z| < R \} \subset A_R(f)^c \cap f^N(U),
\]
it follows that $\partial A_R(f) \cap f^N(U) \neq \emptyset$ and so $\partial A_R(f) \nsubseteq J(f)$. This example also shows that, in Theorem~\ref{main2} part (a), $A_R^{L-1}(f)$ cannot be replaced by $A_R^L(f)$.

\section{The Julia set and $A(f)$}
\setcounter{equation}{0}

   In this section we consider points in $A(f)$ that also belong to the Julia set and prove Theorem~\ref{main3}.

We begin by showing that $A(f)$ has properties corresponding to the properties of $I(f)$ proved by Eremenko~\cite{E} that we listed in~\eqref{E1}. Proofs of parts of this result appeared in~\cite{BH99} and in~\cite{RS05}. The proof of part (a) given here is particularly straightforward.

\begin{theorem}\label{EJ}
Let $f$ be a {\tef}. Then
\begin{itemize}
\item[(a)] $A(f) \cap J(f) \neq \emptyset$;
\item[(b)] $J(f) = \partial A(f)$;
\item[(c)] $J(f) = \overline{A(f) \cap J(f)}$.
\end{itemize}
\end{theorem}
\begin{proof}
{\it (a)} We saw in Theorem~\ref{E} that Eremenko's construction shows that there are always points in $A(f)$. Further, if such a point is in a component $U$ of $F(f)$, then it follows from Theorem~\ref{main2} that $\partial U \subset A(f)$. Since $\partial U$ is also in $J(f)$, this shows that $A(f) \cap J(f) \neq \emptyset$.

{\it (b)} We begin by noting that it follows from Theorem~\ref{main2} that $\partial A(f) \subset J(f)$.

 We now show that $J(f) \subset \partial A(f)$, by following the argument used by Eremenko to show that $J(f) = \partial I(f)$. First note that $A(f)$ is completely invariant and so, for $n \in \N$, $f^{-n}(\overline{A(f)}) \subset \overline{A(f)}$ and hence
\[
 f^n(\C \setminus \overline{A(f)}) \subset \C \setminus \overline{A(f)}.
\]
Since $A(f)$ contains infinitely many points, it follows from Montel's theorem that $\C \setminus \overline{A(f)} \subset F(f)$ and so $J(f) \subset \overline{A(f)}$. Since $A(f)$ contains no periodic points, all interior points of $A(f)$ must belong to $F(f)$. Thus $J(f) \subset \partial A(f)$ and hence $J(f) = \partial A(f)$ as claimed.

{\it (c)} Similarly, since $A(f) \cap J(f)$ is completely invariant and contains infinitely many points, $J(f) \subset \overline{A(f) \cap J(f)}$. Clearly, since $J(f)$ is closed, $\overline{A(f) \cap J(f)} \subset J(f)$ and so $\overline{A(f) \cap J(f)} = J(f)$.
\end{proof}

We saw in the previous section that there exist transcendental entire functions $f$ such that $A(f) \cap F(f) \neq \emptyset$. There are many functions, however, for which $A(f) \subset J(f)$. In particular, this is the case for functions in the Eremenko-Lyubich class
\[
 \B = \{f: f \mbox{ is transcendental entire and } \sing (f^{-1}) \mbox{ is bounded}\},
\]
where the set $ \sing (f^{-1})$ consists of the critical values and the finite asymptotic values of $f$; indeed it was shown by Eremenko and Lyubich~\cite{EL92} that $I(f) \subset J(f)$ if $f \in \B$. For many functions in the class $\B$, the escaping set consists of a family of curves to infinity; see~\cite{RRRS} and~\cite{kB07}. In~\cite{RRS} we showed that, in many cases, $A(f)$ consists of these curves with the exception of some of the endpoints. In particular, this is the case if $f$ is a finite composition of functions of finite order in the class $\B$.

Note that, for such functions, it is possible for both $A(f)$ and $I(f)$ to be connected. We now show that this is the case for the function
\[
 f(z) = \cosh^2 z.
\]
This function has order 1 and is in the class $\B$ since its only singular values are the critical values $0$ and $1$, which are taken at odd and even multiples of $i\pi/2$ respectively. Both singular values are in $I(f)$ so $J(f) = \C$ by~\cite[Theorems 7 and 12]{wB93}.

This function has the properties that $f$ is $\pi i$-periodic, $M(r) = f(r)$, for $r>0$,  and $f(iy) \in [-1,1]$, for $y \in \R$. So, if we set $R=1$, then $M(r) > r$ for $r\geq 1$ and, for $n \in \Z$,
\[
 \{x+iy: |x| \geq 1, y = n\pi\} \subset A_R(f),
\]
\[
\{x+iy: |x| < 1, y = n\pi\} \subset A_R^{-1}(f),
\]
\[
\{x+iy: x = 0\} \subset A_R^{-2}(f).
\]
Thus
\[
 E = \bigcup_{n \in \Z} \{x+iy: y = n\pi\} \cup \{x+iy : x=0\} \subset A_R^{-2}(f).
\]
Clearly $E$ is connected and contains $\{z: f(z) = a\}$, for every $a \geq 1$.

We now note that $f$ maps each half-strip of the form
\[
S_n = \{x+iy: x>0, n\pi < y < (n+1)\pi\}, n \in \Z,
\]
univalently onto the cut plane $\C \setminus \{x: x\geq 0\}$, with each vertex of $S_n$ mapping to $1$ and $i(n + 1/2)\pi$ mapping to $0$. Further, $f$ has a continuous extension to $\partial S_n$, mapping it onto $\{x: x\geq 0\}$. Therefore, for each $n \in \Z$, the set $f^{-1}(E) \cap \overline{S_n}$ forms a connected subset of $A_R^{-3}(f)$ and $\partial S_n \subset E \cap f^{-1}(E)$, so $E \cup f^{-1}(E)$ is connected.

Repeating this argument gives that $T = \bigcup_{n \geq 0} f^{-n}(E)$ is a connected subset of $A(f)$ which contains the backwards orbit of any point in $[1,\infty)$. Since $J(f) = \C$ and the backwards orbit of any non-exceptional point is dense in $J(f)$, it follows that $\overline{T} = \C$. Thus
\[
 T \subset A(f) \subset I(f) \subset \C = \overline{T},
\]
and hence both $A(f)$ and $I(f)$ are connected.

In Figure 2, the set of black points is an approximation to part of the set $A(f)$ for this function, consisting of curves that are successive pre-images under $f$ of the positive real axis. The range shown is given by $|\re z|, |\im z| \leq 2$. The original motivation behind this example was to find a function $f$ for which there is an unbounded continuum meeting only finitely many levels of $A(f)$ -- the imaginary axis is such an unbounded continuum when $f(z) = \cosh^2z$.

\begin{figure}[htb]
\begin{center}
\includegraphics[width=8cm]{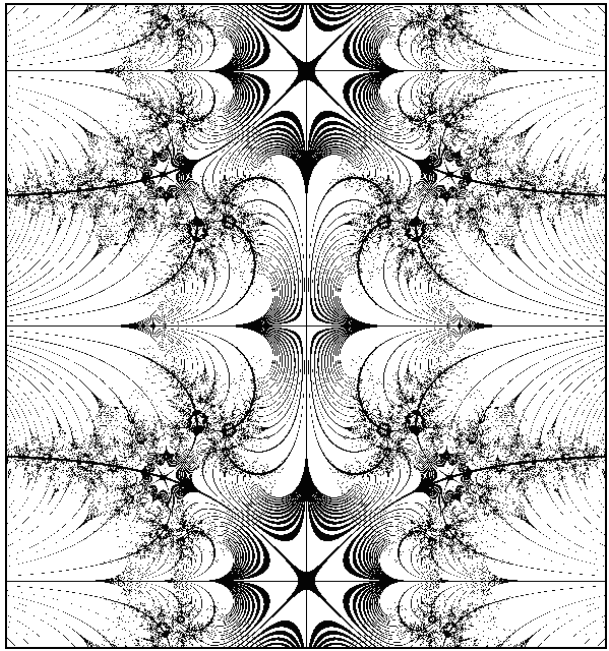}
\caption{Part of $A(f)$ for $f(z)=\cosh^2 z$}
\end{center}
\end{figure}

There are now known to be many examples of functions in the class $\B$ for which $A(f)$ and $I(f)$ are both connected. For example, Rempe~\cite{lR10} has shown that $I(f)$ is connected if $f(z) = e^z + a$, where $a \in (-1,\infty)$, and it follows from his proof that $A(f)$ is also connected for these functions. These results have recently been extended to larger classes of exponential functions in~\cite{Ja} and~\cite{lR10b}.

We now prove the main result of this section.

\begin{proof}[Proof of Theorem~\ref{main3}]
We must show that all the components of $A_R^{L}(f)
\cap J(f)$ are unbounded if and only if $f$ has no multiply connected Fatou components.

We begin by noting that if $U$ is a multiply connected Fatou component of $f$, then it follows from Lemma~\ref{Baker} that there are no unbounded components of $J(f)$.

Now suppose that $f$ has no multiply connected Fatou components and let $z_0 \in A_R^L(f) \cap J(f)$. We know from Theorem~\ref{main1} that $z_0$ belongs to a closed unbounded component $\Gamma$ of $A_R^L(f)$. Further, it follows from Corollary~\ref{AJF} and the remark following it that
\begin{equation}\label{GJ}
\partial \Gamma = \Gamma \cap J(f).
 \end{equation}
We now let $C$ denote the component of $\partial \Gamma$ containing $z_0$ and show that, if $C$ is bounded, then we get a contradiction.

So suppose that $C$ is bounded. Since $\partial \Gamma$ is closed, there exists a Jordan curve $\gamma \subset \C \setminus \partial \Gamma$ that surrounds $C$; see~\cite[page 143]{New}. Since $\Gamma$ is unbounded and connected, and $C \subset \Gamma$, we have $\gamma \cap \Gamma \neq \emptyset$. Since $\gamma \cap \partial \Gamma = \emptyset$, we have $\gamma \subset \Gamma \setminus \partial \Gamma$ and so, by~\eqref{GJ}, $\gamma \subset F(f)$. Also, by~\eqref{GJ}, $C \subset J(f)$ and so $\gamma$ belongs to a multiply connected Fatou component. We have assumed, however, that there are no such components. Thus $C$ must in fact be unbounded. Since
\[
z_0 \in C \subset \partial \Gamma \subset A_R^L(f) \cap J(f),
\]
this completes the proof.
\end{proof}

\section{Spider's web fast escaping sets}
\setcounter{equation}{0}

In the final three sections, we consider the case when $A(f)$ is connected and, moreover, is a spider's web. Recall from the introduction that a connected set $E$ is a spider's web if there
exists a sequence of bounded simply connected domains $G_n$ with $G_n \subset G_{n+1}$,
$\partial G_n \subset E$, for $n \in \N$, and $ \bigcup_{n \in \N}G_n = \C$.

In this section we prove Theorem~\ref{main4} which states that, if $R>0$ is such that $M(r,f) > r$ for $r\geq R$, and $A_R(f)^c$ has a bounded component, then each of $A_R(f)$, $A(f)$ and $I(f)$ is a {\sw}.

\begin{proof}[Proof of Theorem~\ref{main4}]
 Suppose that there exists a bounded component $G$ of $A_R(f)^c$. Then $G$ is simply connected (since all the components of $A_R(f)$ are unbounded by Theorem~\ref{main1}). Further, $\partial G \subset A_R(f)$, since $A_R(f)$ is closed. Now let $G_n = \widetilde{f^n(G)}$. Then, for each $n \in \N$, it follows from~\eqref{A1} and~\eqref{subset} that
\begin{equation}\label{loop}
\partial G_n \subset f^n(\partial G) \subset A_R^n(f) \subset \{z: |z| \geq M^n(R)\}.
 \end{equation}
 Further, if $n$ is sufficiently large, then $\partial G_n$ must surround $0$. Otherwise, for all $n \in \N$, we would have $G_n \subset \{z: |z| \geq M^n(R)\}$ thus contradicting the fact that $G \subset A_R(f)^c$.

   Since $A_R^n(f) \subset A_R(f)$ for $n \in \N$, it follows from~\eqref{loop} that there exists a sequence $(n_k)$ such that $G_{n_k} \subset G_{n_{k+1}}$, $\partial G_{n_k} \subset A_R(f)$ and $\bigcup_{k \in \N}G_{n_k} = \C$. Since each component of $A_R(f)$ is unbounded, we see that each component of $A_R(f)$ will meet $\partial G_{n_k}$ for $k$ sufficiently large. Thus $A_R(f)$ must be connected and indeed must be a spider's web.
Similarly, since each component of $A(f)$ is unbounded, $A(f)$ must be a {\sw}.

  It remains to show that $I(f)$ is a {\sw}. First note that $A(f)$ is connected and is therefore contained in a component $I_0$ of $I(f)$. Further, $I_0$ is a spider's web.
We now show that $I(f) = I_0$.

First suppose that $z_0 \in I(f) \cap J(f)$. It follows from Lemma~\ref{EJ} part (b) that $z_0 \in \partial A(f) \subset \overline{I_0}$. Thus $I_0 \cup \{z_0\}$ is connected. Since $z_0 \in I(f)$ and $I_0$ is a component of $I(f)$, it follows that $z_0 \in I_0$. Thus
\begin{equation}\label{IJ}
I(f) \cap J(f) \subset I_0.
\end{equation}

To complete the proof we show that $I(f) \cap F(f) \subset I_0$. First, we show that, if $V$ is a Fatou component in $I(f)$, then $\partial V \cap I(f) \neq \emptyset$. We can assume that $f^n(V)$ lies in a simply connected component of $F(f)$ for each $n \in \N$ since otherwise $\partial V \subset A(f)$ by Theorem~\ref{Baker1}.

If $\partial V \subset
I(f)$, then the claim is proved. Otherwise, there exist $z_1 \in
\partial V$, $m \in \N$ and a sequence $(n_k)$ such that $f^{n_k}(z_1) \in G_m$, for $k \in \N$, and $G_m \cap J(f) \neq \emptyset$.
Since $V \subset I(f)$, there exists $k \in \N$ such that
$f^{n_k}(V) \cap \partial G_m \neq \emptyset$. Since
$f^{n_k}(V)$ lies in a simply connected component of $F(f)$, there
are no curves in $f^{n_k}(V)$ surrounding points in $J(f)$. So
\[
\partial f^{n_k}(V) \cap \partial G_m \neq \emptyset.
\]
Thus $\partial f^{n_k}(V) \cap I(f) \neq \emptyset$. Since
$\partial f^{n_k}(V) \subset f^{n_k}(\partial V)$, we have
$\partial V \cap I(f) \neq \emptyset$ by the complete invariance
of $I(f)$.

So, whenever $V$ is a Fatou component in $I(f)$, we have
$\partial V \cap I(f)~\neq~\emptyset$. Since $\partial V \subset
J(f)$, it follows from~\eqref{IJ} that $\partial V \cap I_0 \neq
\emptyset$. Now $V \cup (\partial V \cap I_0)$ is a connected
subset of $I(f)$ that meets $I_0$ and so $V \subset I_0$. Hence $I(f) \cap F(f)
\subset I_0$. Together with~\eqref{IJ}, this shows that $I(f) = I_0$ and so $I(f)$ is a spider's web.
\end{proof}

The next result follows easily from Theorem~\ref{main4} and proves Theorem~\ref{main8} part~(a).

\begin{corollary}\label{mc}
Let $f$ be a {\tef} and let $R>0$ be such that $M(r,f) > r$ for $r\geq R$. If $f$ has a multiply connected Fatou component, then each of $A_R(f)$, $A(f)$ and $I(f)$ is a {\sw}.
\end{corollary}
\begin{proof}
If $f$ has a multiply connected Fatou component $U$, then it follows from Theorem~\ref{Baker1} that, for $n$ sufficiently large, $f^n(U) \subset A_R(f)$. Thus, by Lemma~\ref{Baker}, $A_R(f)^c$ has a bounded component. The result then follows from Theorem~\ref{main4}.
\end{proof}

 We give more examples of functions for which $A_R(f)$, $A(f)$ and $I(f)$ are spiders' webs in Section~8. In fact, all the examples of functions for which we know that either $A(f)$ or $I(f)$ is a spider's web also have the property that $A_R(f)$ is a spider's web. This suggests the
following questions.\\

{\bf Question 2} Can $A(f)$ be a spider's web when $A_R(f)$ is
not a spider's web?\\

{\bf Question 3} Can $I(f)$ be a spider's web when $A(f)$ is not
a spider's web?\\

\section{Properties of spiders' webs}
\setcounter{equation}{0}

In this section we show that, if $A_R(f)$ is a spider's web, then several strong results about the structure of the escaping set hold, namely Theorem~\ref{main6}, Theorem~\ref{holes} and Theorem~\ref{main7}. We begin by proving some basic properties of such spiders' webs.

\begin{lemma}\label{SW1}
Let $f$ be a {\tef}, let $R>0$ be such that $M(r,f) > r$ for $r\geq R$, and let $L \in \Z$.
\begin{itemize}

\item[(a)]  If $G$ is a bounded component of $A_R^L(f)^c$, then $\partial G \subset A_R^L(f)$ and $f^n$ is a proper map of $G$ onto a bounded component of $A_R^{n+L}(f)^c$, for each $n \in \N$.
\item[(b)] If $A_R^L(f)^c$ has a bounded component, then $A_R^L(f)$ is a spider's web and hence every component of $A_R^L(f)^c$ is bounded.
\item[(c)] $A_R(f)$ is a spider's web if and only if $A_R^L(f)$ is a spider's web.
\item[(d)] Let $R'>R$. Then $A_R(f)$ is a spider's web if and only if $A_{R'}(f)$ is a spider's web.

\end{itemize}
\end{lemma}
\begin{proof}
{\it (a)} Let $G$ be a bounded component of $A_R^L(f)^c$. Since $A_R^L(f)$ is closed, $\partial G \subset A_R^L(f)$. Thus, for each $n \in \N$, $f^n(G) \subset A_R^{L+n}(f)^c$ and $f^n(\partial G) \subset A_R^{L+n}(f)$. It follows that $\partial f^n(G) = f^n(\partial G)$, so $f^n$ is a proper map of $G$ onto $f^n(G)$ and $f^n(G)$ is a bounded component of $A_R^{L+n}(f)^c$.

{\it (b)} Let $G$ be a bounded component of $A_R^L(f)^c$. Then, as in part (a), $\partial f^n(G) \subset A_R^{L+n}(f)$ and, for $n$ sufficiently large, $\partial f^n(G)$ surrounds $0$, since $f^n(G) \subset A_R^{L+n}(f)^c$. Since each component of $A_R^L(f)$ is unbounded by Theorem~\ref{main1}, it now follows (as in the proof of Theorem~\ref{main4}) that $A_R^L(f)$ is a spider's web. The final assertion of part (b) follows from the definition of a spider's web.

{\it (c)} Suppose that $A_R(f)$ is a spider's web. If $L<0$ then any component of $A_R^L(f)^c$ is contained in a component of $A_R(f)^c$ and is therefore bounded. So, by part~(b), $A_R^L(f)$ is a spider's web. If $L>0$ and $G$ is a component of $A_R(f)^c$ then, by part~(a), $f^L(G)$ is a bounded component of $A_R^L(f)^c$ and so, by part~(b), $A_R^L(f)$ is a spider's web. The converse result follows by similar arguments.

{\it (d)} Suppose that $A_{R'}(f)$ is a {\sw}. Since $R'>R$, each component of $A_{R}(f)^c$ is contained in a component of $A_{R'}(f)^c$, so it follows from part (b) that $A_{R}(f)$ is also a {\sw}. Now suppose that $A_{R}(f)$ is a {\sw}. There exists $L > 0$ such that $M^L(R) \geq R'$ and so each component of $A_{R'}(f)^c$ is contained in a component of $A_{M^L(R)}(f)^c = A_R^L(f)^c$. It then follows from parts (b) and (c) that $A_{R'}(f)$ is also a {\sw}.
\end{proof}

In order to prove the main results of this section we introduce some terminology.

\begin{definition}
Let $f$ be a {\tef} and let $R>0$ be such that $M(r,f) > r$ for $r\geq R$. If $A_R(f)$ is a {\sw} then, for each $n \geq 0$,
\begin{itemize}
\item $H_n$ denotes the component of $A_R^n(f)^c$ that contains 0;

\item$L_n = \partial H_n$.
\end{itemize}
  We say that $(H_n)_{n \geq 0}$ is the {\it sequence of fundamental holes} for $A_R(f)$ and $(L_n)_{n \geq 0}$ is the {\it sequence of fundamental loops} for $A_R(f)$.
\end{definition}

Before proving the main results of this section, we prove some preliminary results about these holes and loops.

\begin{lemma}\label{SW3}
Let $f$ be a {\tef} and let $R>0$ be such that $M(r,f) > r$ for $r\geq R$. Suppose that $A_R(f)$ is a {\sw} and $(H_n)$ and $(L_n)$ are the sequences of fundamental holes and loops for $A_R(f)$. Then
\begin{itemize}
\item[(a)]
$H_n \supset \{z: |z| < M^n(R)\}$ and $L_n \subset A_R^n(f)$, for $n \geq 0$;
\item[(b)] $H_{n+1} \supset H_n$, for $n \geq 0$;
\item[(c)]  for $n \in \N$ and $m \geq 0$,
\[
f^n(H_m) = H_{m+n}, \; f^n(L_m) = L_{m+n};
\]

\item[(d)] there exists $N \in \N$ such that, for $n \geq N$ and $m \geq 0$,
\[
  L_m \cap L_{n+m} = \emptyset;
\]

\item[(e)] if $L \in \Z$ and $G$ is a component of $A_R^L(f)^c$, then $f^n(G) = H_{n+L}$ and $f^n(\partial G) = L_{n+L}$, for $n$ sufficiently large;
\item[(f)] if there are no multiply connected Fatou components, then $L_n \subset J(f)$ for $n \geq 0$.
\end{itemize}
\end{lemma}
\begin{proof}
{\it (a)} This follows from~\eqref{A1} and Lemma~\ref{SW1} part (a).

{\it (b)} Let $n \geq 0$. We know that $H_n \subset A_R^n(f)^c \subset A_R^{n+1}(f)^c$. Also, $H_n$ contains $0$. So, by definition, $H_n \subset H_{n+1}$.

{\it (c)} Let $n \in \N$ and $m \geq 0$. First note that $0 \in H_m$ and that $|f^n(0)| < M^n(R) \leq M^{n+m}(R)$, so $f^n(0) \in H_{n+m}$. The result now follows from Lemma~\ref{SW1} part (a).

{\it(d)} Since $L_0$ is bounded, it is clear that there exists $N \in \N$ such that $L_0 \subset \{z: |z| < M^N(R)\}$ and hence $L_0 \cap L_n = \emptyset$ for $n \geq N$. The result now follows from part~(c).

{\it (e)} If $L \in \Z$ and $G$ is a component of $A_R^L(f)^c$, then it follows from Lemma~\ref{SW1} part~(a) that, for each $n \in \N$, $f^n(G)$ is a component of $A_R^{n+L}(f)^c$. Thus, for each $n \in \N$,  either $f^n(G) = H_{n+L}$ or $f^n(G)$ lies outside $L_{n+L}$, in which case
\begin{equation}\label{Hn}
f^n(G) \subset \{z: |z| \geq M^{n+L}(R)\}.
\end{equation}
Since $G \subset A_R^L(f)^c$,~\eqref{Hn} must fail for large values of $n$. Thus, if $n$ is sufficiently large, $f^n(G) = H_{n+L}$ and so $f^n(\partial G) = L_{n+L}$, by Lemma~\ref{SW1} part~(a).

{\it (f)} Note that, for each $n \geq 0$, $L_n \subset \partial A_R^n(f)$. If there are no multiply connected Fatou components, then it follows from Corollary~\ref{AJF} that, for each $n \geq 0$, $\partial A_R^n(f) \subset J(f)$ and so $L_n \subset J(f)$.
\end{proof}

We are now in a position to prove the main results of this section.

\begin{proof}[Proof of Theorem~\ref{main6}]
Let $f$ be a {\tef}, let $R>0$ be such that $M(r,f) > r$ for $r\geq R$, and suppose that $A_R(f)$ is a spider's web.

{\it (a)} If $f$ has no multiply connected Fatou components, then it follows from Lemma~\ref{SW3} part (f) that
 \begin{equation}\label{l}
 L_n \subset A_R(f) \cap J(f) \subset A(f) \cap J(f) \subset I(f) \cap J(f) \subset J(f),
 \end{equation}
 where $(L_n)$ is the sequence of fundamental loops for $A_R(f)$.
  Since each component of $A_R(f) \cap J(f)$ is unbounded, by Theorem~\ref{main3}, this implies that $A_R(f) \cap J(f)$ is a {\sw}. Similarly, $A(f) \cap J(f)$ must be a {\sw}. It follows from Theorem~\ref{EJ} part~(c) that
\[
 A(f) \cap J(f) \subset I(f) \cap J(f) \subset J(f) = \overline{A(f) \cap J(f)}.
\]
Thus each of $I(f) \cap J(f)$ and $J(f)$ is also connected and moreover, by~\eqref{l}, is a {\sw}.

{\it (b)} We must show that $f$ has no unbounded Fatou components. If $f$ has a multiply connected Fatou component, then this follows from Lemma~\ref{Baker}. If $f$ has no multiply connected Fatou components, then the result follows from part~(a).
\end{proof}

Next we prove Theorem~\ref{holes} which concerns the components of $A(f)^c$.

\begin{proof}[Proof of Theorem~\ref{holes}]
Let $f$ be a {\tef}, let $R>0$ be such that $M(r,f) > r$ for $r\geq R$, and suppose that $A_R(f)$ is a spider's web.

{\it (a)} Let $K$ be a component of $A(f)^c$. We must show that $K$ is compact.  For each $L \in \N$, we
let $G_L$ denote the component of $A_R^{-L}(f)^c$ that contains $K$. Clearly $K \subset \bigcap_{L \in \N} G_L \subset A(f)^c$. Moreover, it follows from Lemma~\ref{SW3} parts (d) and (e) that there exists $N \in \N$ such that
\[
\overline{G_{L+N}} \subset G_L \subset \overline{G_L},
\]
for $L \in \N$, and so
\begin{equation}\label{G1}
 \bigcap_{L \in \N} G_L = \bigcap_{L \in \N} \overline{G_L}.
\end{equation}
Thus $\bigcap_{L \in \N} G_L$ is a compact connected set in $A(f)^c$. Since $K \subset \bigcap_{L \in \N} G_L$ and $K$ is a component of $A(f)^c$, it follows that
\begin{equation}\label{G}
K = \bigcap_{L \in \N} G_L
\end{equation}
and that $K$ is compact.

{\it (b)} We now show that every point $z_0 \in J(f)$ is the limit of a sequence of points, each of which lies in a distinct component of $A(f)^c$. We begin by considering the case that $z_0 \in J(f) \cap A(f)$. We know from Theorem~\ref{EJ} part~(b) that $J(f) = \partial A(f)$ and so there exists a sequence $(z_n)$ of points in $A(f)^c$ with $z_0 = \lim_{n \to \infty} z_n$. If the points $z_n$ belong to only finitely many components of $A(f)^c$ then, by part~(a), $z_0$ also belongs to one of these components and is therefore in $A(f)^c$. This, however, is a contradiction. So the result holds if $z_0 \in J(f) \cap A(f)$.

We now consider the case that $z_0 \in J(f) \setminus A(f)$. We know from Theorem~\ref{EJ} part (c) that $J(f) = \overline{A(f) \cap J(f)}$. So there exists a sequence $(z'_n)$ of points in $A(f) \cap J(f)$ with $z_0 = \lim_{n \to \infty}z_n'$. We know that the result holds for each point $z_n'$ and this is sufficient to show that the result also holds for $z_0$.
\end{proof}

We now prove Theorem~\ref{main7}, which gives various subsets of $I(f)$ on which the iterates of $f$ escape to infinity uniformly.

\begin{proof}[Proof of Theorem~\ref{main7}]
Let $f$ be a {\tef}, let $R>0$ be such that $M(r,f) > r$ for $r\geq R$, and suppose that $A_R(f)$ is a spider's web.

Let $z_0 \in I(f)$. We show that $z_0$ belongs to an unbounded continuum in $I(f)$ on which all points
escape to infinity uniformly. If $z_0 \in A_R^L(f)$, for some $L \in \Z$, then we know from Theorem~\ref{main1} that $z_0$ belongs to an unbounded component of $A_R^L(f)$. Clearly, all points in this component escape to infinity uniformly.

Now suppose that $z_0 \in I(f) \setminus A(f)$. We let $K$ denote the component of $A(f)^c$ containing $z_0$ and, for $L \in \N$,
let $G_L$ denote the component of $A_R^{-L}(f)^c$ that contains $K$, and let $C_L = \partial G_L$. Note that $C_L$ is connected. It follows from~\eqref{subset} and~\eqref{G} that
\[
G_{L+1} \subset G_L \mbox{ and } K = \bigcap_{L \in \N} G_L.
\]
We now let
\[
 K_1 = A_R^{-1}(f) \mbox{ and } K_{L+1} = A_R^{-(L+1)}(f) \cap \overline{G_L}, \mbox{ for each } L \in \N.
\]
Note that
\begin{equation}\label{conn}
C_L \subset K_L \cap K_{L+1}, \mbox{ for each } L \in \N.
\end{equation}

We claim that
\[
C = K \cup \bigcup_{L \in \N}K_L
\]
is an unbounded continuum in $I(f)$ on which all points escape to infinity uniformly. We begin by showing that $C$ is an unbounded continuum. First,  it follows from Lemma~\ref{SW1} part (c) that $A_R^{-L}(f)$ is a {\sw} and hence an unbounded continuum, for each $L \in \N$. In particular, $K_1$ is an unbounded continuum.

We now claim that $K_{L+1}$ is connected, for $L \in \N$. This is the case because a separation of $K_{L+1}$ would lead to a separation of $A_R^{-(L+1)}(f)$, since $K_{L+1} = A_R^{-(L+1)}(f) \cap \overline{G_L}$ and $\partial G_{L} \subset A_R^{-(L+1)}(f)$. But there can be no separation of $A_R^{-(L+1)}(f)$ since this set is connected.

Further, it follows from~\eqref{conn} that $K_L \cup K_{L+1}$ is connected, for each $L \in \N$. Thus $\bigcup_{L \in \N}K_L$ is unbounded and connected.
It follows that $K \cup \bigcup_{L \in \N} K_L$ is connected because any open set that contains $K$ must also contain a set $C_L = \partial G_L$ for some $L \in \N$, by~\eqref{G1} and~\eqref{G}, and so meet $K_L$. Thus
$C = K \cup \bigcup_{L \in \N}K_L$ is an unbounded continuum as claimed.

We now show that all the points in $C$ escape to infinity uniformly. Let $(H_n)$ and $(L_n)$ denote the sequences of fundamental holes and loops for $A_R(f)$. Recall that $z_0 \in K \cap I(f)$ and so, for each $m \in \N$, there exists $N_m \in \N$ such that
\[
f^n(z_0) \in \C \setminus H_m, \; \mbox{ for } n \geq N_m.
\]
We claim that
\begin{equation}\label{image}
 f^n(K \cup \bigcup_{L \in \N} K_{L+1}) \subset \C \setminus H_m, \;  \mbox{ for } n \geq N_m,
\end{equation}
and so all points in $K \cup \bigcup_{L \in \N} K_{L+1}$ escape to infinity uniformly.

To prove this, we let $m \in \N$ and $n \geq N_m$ and consider the images of the sets $G_L$ for $L \in \N$. We begin by noting that it follows from Lemma~\ref{SW1} part (a) and~\eqref{subset} that, for $L \in \N$,
\begin{equation}\label{property1}
 f^n(G_L) \subset A_R^{n-L}(f)^c \mbox{ and } f^n(K_{L+1}) \subset A_R^{n-L-1}(f).
  \end{equation}
  Also
  \begin{equation}\label{property2}
   H_m \subset A_R^m(f)^c \mbox{ and } L_m \subset A_R^m(f).
\end{equation}
By the second relation in~\eqref{property1} and the first in~\eqref{property2} we have
 \begin{equation}\label{K}
 f^n(K_{L+1}) \subset \C \setminus H_m, \; \mbox{ for } n-L > m.
 \end{equation}
 Also, since $n\geq N_m$, we have $f^n(z_0) \in \C \setminus H_m$ and so it follows from~\eqref{property1} and~\eqref{property2} that
 \begin{equation}\label{GL}
  f^n(\overline{G_L}) \subset \C \setminus H_m, \; \mbox{ for } n-L \leq m,
 \end{equation}
 because $z_0 \in G_L$ and $f^n(G_L)$ is contained in a component of $A_R^m(f)^c$ different from $H_m$ for such~$L$.

 Since $K = \bigcap_{L \in \N}G_L$, it follows from~\eqref{GL} that
 \begin{equation}\label{GH}
 f^n(K) \subset \C \setminus H_m, \; \mbox{ for } n \geq N_m.
 \end{equation}
  Since $K_{L+1} \subset \overline{G_L}$, it also follows from~\eqref{GL} that
  \[
 f^n(K_{L+1}) \subset \C \setminus H_m, \; \mbox{ for } n-L \leq m.
 \]
 Together with~\eqref{K} and~\eqref{GH}, this is sufficient to show that~\eqref{image} holds. To conclude the proof, we note that all points escape to infinity uniformly on $K_1 = A_R^{-1}(f)$. Together with~\eqref{image}, this implies that all points escape to infinity uniformly on $C = K \cup \bigcup_{L \in \N}K_L$.
\end{proof}

\section{Examples of spiders' webs}
\setcounter{equation}{0}

We begin this section by proving Theorem~\ref{SW2} which shows in particular that, if $f$ belongs to the widely studied Eremenko-Lyubich class $\B$, then $A_R(f)$ is {\it not} a spider's web.

\begin{proof}[Proof of Theorem~\ref{SW2}]
 If $A_R(f)$ is a spider's web, then there is no path to $\infty$ on which $f$ is bounded, by Lemma~\ref{SW3}~part~(c), since this implies that $|f(z)| \geq M^{m+1}(R)$, for $z \in L_m$, $m \in \N$.

If $f \in \B$, then there must be a path to $\infty$ on which $|f|$ is constant (see~\cite{EL92}), so $A_R(f)$ is not a spider's web. Finally, if $f$ has an exceptional value $\alpha$, then $f(z) = \alpha + (z-\alpha)^p e^{g(z)}$, for some $p \geq 0$ and some entire function $g$ (see~\cite{wB93}), so
\[
h(z) = \frac{1}{f(z) - \alpha} = \frac{e^{-g(z)}}{(z-\alpha)^p}
\]
is a transcendental meromorphic function with one pole, at $\alpha$. Hence $h$ has asymptotic value $\infty$ by Iversen's Theorem (see, for example,~\cite[p.286]{Nev}), so $f$ has asymptotic value $\alpha$ and is therefore bounded on a path to $\infty$.
\end{proof}

 The rest of this section is devoted to showing that there are large classes of functions for which $A_R(f)$ is a spider's web whenever $M(r,f) > r$ for $r\geq R>0$; in particular, we prove parts~(b)--(e) of Theorem~\ref{main8}. (Recall that we proved Theorem~\ref{main8} part (a) at the end of Section 6.) We obtain these examples by using the following result.

\begin{theorem}\label{gamman}
Let $f$ be a {\tef} and let $R>0$ be such that $M(r,f) > r$ for $r\geq R$. Then $A_R(f)$ is a {\sw} if and only if there exists a sequence $(G_n)_{n \geq 0}$ of bounded simply connected domains such that, for all $n \geq 0$,
\begin{equation}\label{gamma1}
G_n \supset \{z: |z| < M^n(R)\}
\end{equation}
and
\begin{equation}\label{gamma2}
  G_{n+1} \mbox{ is contained in a bounded component of } \C \setminus f(\partial G_n).
\end{equation}
\end{theorem}
\begin{proof}
If $A_R(f)$ is a {\sw} and $(H_n)$ is the sequence of fundamental holes for $A_R(f)$, then it follows from Lemma~\ref{SW3} parts (a) and (b) that the sets $G_n = H_{n}$ have properties~\eqref{gamma1} and~\eqref{gamma2}.

Now suppose that there exists a sequence $(G_n)$ of bounded simply connected domains satisfying~\eqref{gamma1} and \eqref{gamma2}, and let $G$ denote the component of $A_R(f)^c$ containing $\{z: |z| < R\}$. We will show that $G$ is bounded and hence, by Lemma~\ref{SW1} part (b), that $A_R(f)$ is a {\sw}. In order to do this, we suppose that $G$ is not contained in $G_0$ and obtain a contradiction. For simplicity, we put $\gamma_n = \partial G_n$, for $n \geq 0$.

If $G$ is not a subset of $G_0$, then there exists a compact connected set $K$ such that
\begin{equation}\label{gamma3}
 K \subset G, \; K \cap \{z:|z|<R\} \neq \emptyset \; \mbox{ and } \; K \cap \gamma_0 \neq \emptyset.
\end{equation}
It follows from~\eqref{gamma2} and~\eqref{gamma3} that $f(K)$ is a compact connected set with
\[
f(K) \cap \{z: |z| < M(R)\} \neq \emptyset \; \mbox{ and } \; f(K) \cap \gamma_1 \neq \emptyset.
\]
Thus, by~\eqref{gamma1},
\[
 K_1 = \{z \in K: |f(z)| \geq M(R)\} \neq \emptyset.
\]
Using~\eqref{gamma1} and~\eqref{gamma2} repeatedly in this way, we deduce that, for each $n \in \N$,
\begin{equation}\label{gamma4}
 K_n = \{z \in K: |f^n(z)| \geq M^n(R)\} \neq \emptyset.
\end{equation}
Since, for each $n \in \N$, $K_n$ is compact and $K_{n+1} \subset K_n$, we have $\bigcap_{n \in \N} K_n \neq \emptyset$. It follows from~\eqref{gamma4} that
\[
\bigcap_{n \in \N}K_n \subset A_R(f)
\]
and from~\eqref{gamma3} that
\[
\bigcap_{n \in \N}K_n \subset K \subset G \subset A_R(f)^c.
\]
This is a contradiction and so our original supposition that $G$ is not a subset of $G_0$ must have been incorrect. Thus $G \subset G_0$ and is therefore bounded as required.
\end{proof}

We obtain the following corollary of Theorem~\ref{gamman} by taking the domains $G_n$ to be discs. Here $m(r) = \min_{|z| = r}|f(z)|$, for $r>0$.

\begin{corollary}\label{m(r)}
Let $f$ be a {\tef} and let $R>0$ be such that $M(r,f) > r$ for $r\geq R$. Then $A_R(f)$ is a {\sw} if there exists a sequence $(\rho_n)$ such that, for $n \geq 0$,
\begin{equation}\label{r1}
\rho_n > M^n(R)
\end{equation}
and
\begin{equation}\label{r2}
m(\rho_n) \geq \rho_{n+1}.
\end{equation}
\end{corollary}

Much work has been done on establishing when the conditions of this corollary are satisfied in connection with a conjecture of Baker that functions of small growth have no unbounded Fatou components. (For further details on the history of this problem see the survey article~\cite{H}.) The strongest results on this problem are given in~\cite{RS08} and~\cite{HM}. Recall from Theorem~\ref{main6} part (b) that if $A_R(f)$ is a spider's web, then $f$ has no unbounded Fatou components. Thus, if $A_R(f)$ is a spider's web, then both Eremenko's conjecture (that all the components of $I(f)$ are unbounded) and Baker's conjecture hold.

It follows from the discussion at the end of the first section of~\cite{RS08} that the conditions of Corollary~\ref{m(r)} are satisfied if $f$ is a {\tef} and there exist $m \geq 2$ and $r_0>0$ such that
\begin{equation}\label{small}
 \log \log M(r) < \frac{\log r}{\log^m r}, \mbox{ for } r > r_0.
\end{equation}
Here $\log^m$ denotes the $m$th iterated logarithm function. Thus, for such functions, $A_R(f)$ is a {\sw}. As discussed in~\cite[Section 6]{RS07}, there are many examples of such functions for which there are no multiply connected Fatou components. Firstly, Bergweiler and Eremenko showed in~\cite{BE} (see also~\cite{BE01}) that there are functions of arbitrarily small growth (and hence functions satisfying~\eqref{small}) for which the Julia set is the whole plane. Further, Baker~\cite{B01} and Boyd~\cite{Bo} independently showed that there are functions of arbitrarily small growth (and hence functions satisfying~\eqref{small}) for which every point in the Fatou set tends to $0$ under iteration.

We now define the order and lower order of a {\tef} since these concepts play a key role in discussing further examples of functions for which $A_R(f)$ is a {\sw}. The {\it order of growth} of a {\tef} $f$ is defined by
\[
  \rho(f) = \limsup_{r \to \infty} \frac{\log \log M(r)}{\log r}
\]
and the {\it lower order of growth} is
\[
 \lambda(f) = \liminf_{r \to \infty} \frac{\log \log M(r)}{\log r}.
\]

 Functions satisfying~\eqref{small} have order zero. It follows from the results in~\cite[Section 6]{RS07} that there are many functions of order less than $1/2$ with regular growth for which $A_R(f)$ is a {\sw}. (Note that we do not use this terminology in that paper -- there we say that for these functions the set $B_D(f)$ is connected. This set has an analogous definition to the set $A_R(f)$ -- see the remarks after Corollary~\ref{equal}.)

 We now give a general result which can be used to show that $A_R(f)$ is a {\sw}, not only for the functions discussed above, but for many other classes of functions.

 \begin{corollary}\label{regular}
 Let $f$ be a {\tef} and let $R>0$ be such that $M(r,f) > r$ for $r\geq R$. Then $A_R(f)$ is a {\sw} if, for some $m>1$,
\begin{itemize}
\item[(a)]there exists $R_0 > 0$ such that, for all $r \geq R_0$,
\begin{equation}\label{min}
 \mbox{ there exists } \rho \in (r,r^m) \mbox{ with } m(\rho) \geq M(r), \mbox{ and }
\end{equation}

\item[(b)] $f$ has regular growth in the sense that there exists a sequence $(r_n)$ with
    \begin{equation}\label{reg}
     r_n > M^n(R) \mbox{ and } M(r_n) \geq r_{n+1}^m, \mbox{ for } n \geq 0.
    \end{equation}
\end{itemize}
 \end{corollary}
 \begin{proof}
 Let $R_0 > 0$ be as in part (a). Then, by part (b), there exists a sequence $(r_n)$ satisfying~\eqref{reg} with $r_n > R_0$, for $n \geq 0$. So, by~\eqref{min}, for each $n \geq 0$, there exists $\rho_n \in (r_n, r_n^m)$ with
 \[
  m(\rho_n) \geq M(r_n) \geq r_{n+1}^m > \rho_{n+1}.
 \]
 By~\eqref{reg}, we have $\rho_n > r_n > M^n(R)$, for $n \geq 0$. Thus the conditions of Corollary~\ref{m(r)} are satisfied and hence $A_R(f)$ is a {\sw}.
 \end{proof}

Note that Corollary~\ref{regular} generalises~\cite[Theorem 5]{RS08} which states that, if $f$ is a {\tef} of order less than 1/2, then $f$ has no unbounded Fatou components whenever, for all $m>1$, there exists a real function
$\psi$ defined on $(r_0, \infty)$, where $r_0>0$, such that, for
$r\geq r_0$, we have $\psi(r) \geq r$ and
\begin{equation}\label{sg}
  M(\psi(r)) \geq \left(\psi(M(r))\right)^m.
\end{equation}
Indeed, if $f$ has order less than $1/2$, then $f$ satisfies Corollary~\ref{regular} part (a) for all sufficiently large values of $m$ -- this was proved by Baker~\cite[Satz 1]{B58} and also follows from the version of the $\cos \pi \rho$ theorem proved by Barry~\cite{Ba}. If the function $\psi$ above exists for such a value of $m$, then Corollary~\ref{regular} part (b) holds by taking $r_n = \psi(M^n(R))$. Then Corollary~\ref{regular} implies that $A_R(f)$ is a spider's web, so $f$ has no unbounded Fatou components by Theorem~\ref{main6} part (b).

 For a discussion of various functions satisfying the regularity condition~\eqref{sg}, see~\cite[Section 7]{RS08}. For example, we show there that~\eqref{sg} holds for a suitable function $\psi$
 if there exists $c>0$ such
that the function $\phi(x) = \log M(e^x)$ satisfies
\begin{equation}\label{AHr}
  \frac{\phi'(x)}{\phi(x)} \geq \frac{1+c}{x}, \; \mbox{ for large } x.
\end{equation}
 The condition~\eqref{AHr} is used by Anderson and Hinkkanen in~\cite{AH} and it is satisfied if $f$ has finite order and positive lower order -- see~\cite[p.205]{H}. So, if $f$ has finite order and positive lower order, then $f$ satisfies Corollary~\ref{regular} part (b).

We now discuss several classes of functions for which Corollary~\ref{regular} part (a) is satisfied for all sufficiently large values of $m$. It follows from Corollary~\ref{regular} that, for functions in these classes,  $A_R(f)$ is a {\sw} whenever the growth is regular, as defined by~\eqref{reg} for an appropriate value of $m$.\\

{\bf Class 1}

As mentioned above, it follows from the $\cos \pi \rho$ theorem that any function of order less than $1/2$ satisfies Corollary~\ref{regular} part (a) for all sufficiently large values of $m$. Many examples of functions of order less than $1/2$ with regular growth are given in~\cite[Section 6]{RS07}. These examples include the function
\[
f(z)=\frac12\left(\cos z^{1/4} + \cosh z^{1/4}\right)=1+\frac{z}{4!}+\frac{z^2}{8!}+\cdots,
\]
which has order $1/4$.
The spider's web $A_R(f)$ for this function is illustrated in Section 1.\\

{\it Remark.} Rod Halburd (personal communication) has pointed out that, more generally, functions of the form
\[
f(z) = \sum_{n=0}^{\infty} \frac{z^{pn}}{(qn)!}, \; \; p,q \in \N,
\]
are entire with order $p/q$. These functions also have regular growth, so if $p/q < 1/2$, then $A_R(f)$ is a spider's web. Further to this, Sixsmith~\cite{Si} has observed that, for many entire functions $g(z) = \sum_{n=0}^{\infty} a_nz^n$, the function
\[
 f(z) = \sum_{n=0}^{\infty} a_{qn}z^{pn}, \; \; p,q \in \N,
\]
has an $A_R(f)$ spider's web.\\

{\bf Class 2}

Many entire functions for which the power series expansion has wide enough `gaps' also satisfy Corollary~\ref{regular} part (a) for all sufficiently large values of $m$. A {\tef} $f$ has {\it Fabry gaps} if
\[
 f(z) = \sum_{k=0}^{\infty} a_k z^{n_k}
\]
 and $n_k/k \to \infty$ as $k \to \infty$. It was shown by Fuchs~\cite{F} that, if $f$ has finite order and Fabry gaps, then, for each $\eps > 0$,
 \[
  \log m(r) > (1-\eps) \log M(r)
 \]
 holds for all $r$ outside a set $E$ for which
 \[
  \lim_{r \to \infty} \frac{1}{\log r} \int_{E \cap (1,r)} \frac{dt}{t} = 0.
 \]
 Together with Lemma~\ref{convex}, this result shows that functions of finite order with Fabry gaps satisfy Corollary~\ref{regular} part (a) for each $m>1$. We now give an example of such a function $f$ for which Corollary~\ref{regular} part (b) is also satisfied for each $m>1$, and hence $A_R(f)$ is a {\sw}.

 \begin{example}
 Let $f$ be the {\tef} given by
 \[
 f(z) = \sum_{k=0}^{\infty}\frac{1}{(k^2)!} z^{[ck]^2}, \mbox{ where } c>0.
 \]
 Then $f$ is a function of order $c^2$ with Fabry gaps and regular growth, and so $A_R(f)$ is a {\sw}.
 \end{example}
 \begin{proof}
Note that, for $c \in \N$, the power series defining $f$ is obtained from the Taylor series for $\exp(z^{c^2})$ by discarding all terms except those for which the power of $z$ is of the form $c^2k^2$.  The function $f$ clearly has Fabry gaps, and it is not difficult to check that it has order and lower order equal to $c^2$. These can be calculated either directly or by using the following formulae (see~\cite{J}) for the order and lower order of an entire function of the form $f(z) = \sum_{n=0}^{\infty} a_n z^n$:

\begin{equation}\label{order}\rho(f)=\limsup_{n\to\infty}\frac{n\log n}{\log
1/|a_n|}\quad\text{and}\quad
\lambda(f)=\max_{(n_k)}\liminf_{k\to\infty}\frac{n_k\log
n_{k-1}}{\log 1/|a_{n_k}|}.
\end{equation}

Since $f$ has finite order and Fabry gaps, it satisfies Corollary~\ref{regular} part (a) for each $m>1$. Since $f$ has finite order and positive lower order, it follows from the discussion after Corollary~\ref{regular} that the growth of $f$ is sufficiently regular to satisfy Corollary~\ref{regular} part (b) for each $m>1$. Thus, by Corollary~\ref{regular}, $A_R(f)$ is a {\sw}.
\end{proof}

 There are also functions of {\it infinite} order with gap series for which Corollary~\ref{regular} part (a) is satisfied for each $m>1$. Hayman~\cite{H72} showed that Fuch's result above, for functions of finite order with Fabry gaps, holds for functions of any order provided that a stronger gap series condition is satisfied, which can be stated as follows:
 \begin{equation}\label{gaps}
  n_k > k \log k(\log \log k)^{\alpha},
 \end{equation}
 for some $\alpha >2$ and sufficiently large values of $k$. Thus, by the same argument as above, such functions also satisfy Corollary~\ref{regular} part (a) for each $m>1$. An explicit example of a {\tef} of infinite order which satisfies~\eqref{gaps} and with sufficiently regular growth to ensure that $A_R(f)$ is a spider's web is given by Sixsmith~\cite{Si}.

 {\it Remark.} It was proved by Wang~\cite{W} that functions with gap series of the type described here  and with regular growth of a certain type have no unbounded Fatou components. This is also implied by our results since functions for which $A_R(f)$ is a {\sw} have no unbounded Fatou components.\\

 {\bf Class 3}

 Many functions which show the pits effect also satisfy Corollary~\ref{regular} part (a) for all sufficiently large values of $m$. Roughly speaking, a function shows the pits effect if it is very large except in small regions, known as pits, around the zeros of~$f$. If the pits are sufficiently small and spaced out and the function is sufficiently large outside the pits, then Corollary~\ref{regular} part (a) is satisfied. The pits effect has been studied by many authors and, in most cases, it is sufficiently strong for Corollary~\ref{regular} part (a) to be satisfied.

 Here we look in detail at the functions studied by Littlewood and Offord in~\cite{LO} since it follows from their results that, in a certain sense, for almost all functions of finite order and positive lower order, $A_R(f)$ is a {\sw}. Thus functions for which $A_R(f)$ is a spider's web are not unusual, as was originally thought.  Other papers giving examples of functions which show a suitably strong version of the pits effect include~\cite{EO} and~\cite{O}.

 \begin{example}
 Suppose that
 \[
  \sum_{n=0}^{\infty}a_nz^n
 \]
 is a {\tef} of order $\rho \in (0,\infty)$ with positive lower order and that
 \[
  C = \{f: f(z) = \sum_{n=0}^{\infty} \eps_n a_n z^n\},
 \]
 where the $\eps_n$ take the values $\pm 1$ with equal probability.
 Then $A_R(f)$ is a {\sw} for almost all functions in $C$.
 \end{example}
 \begin{proof}
 It was shown by Littlewood and Offord in~\cite{LO} that almost all functions in the set $C$ show the pits effect. More precisely, they proved, among other things, that, for each $\eps > 0$, there exist $A,B,r_0>0$ and a set $C_{\eps} \subset C$ of measure less than $\eps$ such that if $f \in C \setminus C_{\eps}$, $r>r_0$, $\mu(r) = \sup_{n}|a_n|r^n$ and $N(r)$ is the largest~$n$ satisfying this equality, then
 \begin{enumerate}
 \item[(a)] if $z$ is not in a pit and $|z| = r$, then
 \begin{equation}\label{pit}
 \log |f(z)| > \log \mu(r) - r^{\eps};
 \end{equation}
 \item[(b)] if $z$ is in a pit and $|z| = r > r_0$, then the radius of the pit is at most $\exp(-r^{\eps})$;
 \item[(c)] the number of pits in the annulus $\{z: r/2 \leq |z| \leq 3r/2\}$ is at most $AN(2r) + Br^{\eps}$.
 \end{enumerate}
 It follows from~\cite[Theorem 6 and the proof of (6.20)]{H1} and also from~\cite[Theorem 6.23 and Lemma 6.15]{H2} that
\begin{equation}\label{N(r)}
N(r) < (\log M(r))^2
\end{equation}
and
\begin{equation}\label{mu(r)}
\mu(r) > (M(r))^{1/2},
\end{equation}
 for $r \notin E$, where $E$ is a set of finite logarithmic measure; that is, $\int_{E \cap [1,\infty)} \frac{dt}{t} < \infty$.
We now take $\eps$ with $0 < \eps < \lambda/2$, where $\lambda$ is the lower order of $f$. It follows from~(a) above together with~\eqref{mu(r)} that, if $z$ is not in a pit and $|z| = r$, then
\begin{equation}\label{large}
 \log |f(z)| > \frac{1}{4} \log M(r),
\end{equation}
for all values of $r$ except those in a set of finite logarithmic measure. Further, it follows from (b) and (c) above together with~\eqref{N(r)} that the set of values of $r$ for which $\{z:|z| = r\}$ meets a pit has finite logarithmic measure and so~\eqref{large} holds except in a set of finite logarithmic measure. Together with Lemma~\ref{convex}, this is sufficient to show that, for each function in $C \setminus C_{\eps}$, Corollary~\ref{regular} part (a) is satisfied for all sufficiently large values of $m$.  The result now follows since $\eps$ can be chosen to be arbitrarily small and functions in the class $C$ have finite order and positive lower order, and hence satisfy Corollary~\ref{regular} part (b) for each $m>1$.
 \end{proof}

{\it Remark.} It is of interest to know which of the above examples of functions with $A_R(f)$ spiders' webs have no multiply connected Fatou components. In forthcoming work~\cite{RSreg} we show that any {\tef} for which $M(r,f)$ satisfies the condition~\eqref{AHr}, introduced by Anderson and Hinkkanen, has {\it no} multiply connected Fatou components, and we give sufficient conditions on $f$ for~\eqref{AHr} to hold.

 We end this section by proving the following result which allows us to construct many more examples of functions for which $A_R(f)$ is a {\sw}. For example, we can deduce from Theorem~\ref{main99} and a theorem of P\'olya~\cite[Theorem 2.9]{wH64} that, if $f$ has non-zero order and $A_R(f)$ is a spider's web, then $f^2$ has infinite order and $A_R(f^2)$ is a spider's web.

 \begin{theorem}\label{main99}
Let $f$ be a {\tef} and let $R>0$ be such that $M(r,f) > r$ for $r\geq R$. Then, for $m \in \N$,  $A_R(f^m)$ is a spider's web if and only if $A_R(f)$ is a {\sw}.
\end{theorem}
\begin{proof}
Suppose that $A_R(f)$ is a {\sw}. It follows from Lemma~\ref{SW1} part (b) that each component of $A_R(f)^c$ is bounded. We know from Theorem~\ref{eq} that $A_R(f) \subset A_R(f^m)$ and so
 each component of $A_R(f^m)^c$ is bounded. Thus, by Lemma~\ref{SW1} part (b), $A_R(f^m)$ is a {\sw}.

To prove the converse, we use the result of Theorem~\ref{eq} that
if $R'\geq R$ is sufficiently large, then $A_{R'}(f^m) \subset A_{R'/2}(f)$.
It follows from Lemma~\ref{SW1} part (d) that, if $A_R(f^m)$ is a {\sw} then $A_{R'}(f^m)$ is also a {\sw} and hence, by Lemma~\ref{SW1} part (b), that each component of $A_{R'}(f^m)^c$ is bounded. Thus each component of $A_{R'/2}(f)^c$ is also bounded and so, by Lemma~\ref{SW1} part (b), $A_{R'/2}(f)$ is a {\sw}. Thus, by Lemma~\ref{SW1} part (d), $A_R(f)$ is a {\sw}.
\end{proof}


{\bf Acknowledgements.} We thank Walter Bergweiler and Lasse Rempe for stimulating discussions, and Dan Nicks, John Osborne and Dave Sixsmith for their helpful comments. We also thank Dominique Fleischmann for producing the picture of the $A_R(f)$ spider's web.

\end{document}